\documentclass[10pt,reqno]{amsart}
\usepackage{latexsym,amsmath,amssymb,amscd}
\usepackage[all]{xy}

\def\today{\ifcase \month \or
   January \or February \or March \or April \or
   May \or June \or July \or August \or
   September \or October \or November \or December \fi
   \space\number\day , \number\year}
   
\makeatletter
  \newcommand\@dotsep{4.5}
  \def\@tocline#1#2#3#4#5#6#7{\relax
     \ifnum #1>\c@tocdepth % then omit
     \else
     \par \addpenalty\@secpenalty\addvspace{#2}%
     \begingroup \hyphenpenalty\@M
     \@ifempty{#4}{%
     \@tempdima\csname r@tocindent\number#1\endcsname\relax
        }{%
         \@tempdima#4\relax
           }%
      \parindent\z@ \leftskip#3\relax \advance\leftskip\@tempdima\relax
      \rightskip\@pnumwidth plus1em \parfillskip-\@pnumwidth
       #5\leavevmode\hskip-\@tempdima #6\relax
       \leaders\hbox{$\m@th
       \mkern \@dotsep mu\hbox{.}\mkern \@dotsep mu$}\hfill
       \hbox to\@pnumwidth{\@tocpagenum{#7}}\par
       \nobreak
        \endgroup
         \fi}
\makeatother 
\begin{document}

%\DeclareRobustCommand{\SkipTocEntry}[4]{} 

\makeatletter
\@addtoreset{figure}{section}
\def\thefigure{\thesection.\@arabic\c@figure}
\def\fps@figure{h,t}
\@addtoreset{table}{bsection}

\def\thetable{\thesection.\@arabic\c@table}
\def\fps@table{h, t}
\@addtoreset{equation}{section}
\def\theequation{%\thesection.
\arabic{equation}}
\makeatother

\newcommand{\bfi}{\bfseries\itshape}

\newtheorem{theorem}{Theorem}
\newtheorem{acknowledgment}[theorem]{Acknowledgment}
\newtheorem{corollary}[theorem]{Corollary}
\newtheorem{definition}[theorem]{Definition}
\newtheorem{example}[theorem]{Example}
\newtheorem{lemma}[theorem]{Lemma}
\newtheorem{notation}[theorem]{Notation}
\newtheorem{proposition}[theorem]{Proposition}
\newtheorem{remark}[theorem]{Remark}
\newtheorem{setting}[theorem]{Setting}

\numberwithin{theorem}{section}
\numberwithin{equation}{section}

\renewcommand{\1}{{\bf 1}}
\newcommand{\Ad}{{\rm Ad}}
\newcommand{\Alg}{{\rm Alg}\,}
\newcommand{\Aut}{{\rm Aut}\,}
\newcommand{\ad}{{\rm ad}}
\newcommand{\Borel}{{\rm Borel}}
\newcommand{\Ci}{{\mathcal C}^\infty}
\newcommand{\Cpol}{{\mathcal C}^\infty_{\rm pol}}
\newcommand{\Der}{{\rm Der}\,}
\newcommand{\de}{{\rm d}}
\newcommand{\ee}{{\rm e}}
\newcommand{\End}{{\rm End}\,}
\newcommand{\ev}{{\rm ev}}
\newcommand{\id}{{\rm id}}
\newcommand{\ie}{{\rm i}}
\newcommand{\GL}{{\rm GL}}
\newcommand{\gl}{{{\mathfrak g}{\mathfrak l}}}
\newcommand{\Hom}{{\rm Hom}\,}
\newcommand{\Img}{{\rm Im}\,}
\newcommand{\Ind}{{\rm Ind}}
\newcommand{\Ker}{{\rm Ker}\,}
\newcommand{\Lie}{\text{\bf L}}
\newcommand{\m}{\text{\bf m}}
\newcommand{\pr}{{\rm pr}}
\newcommand{\Ran}{{\rm Ran}\,}
\renewcommand{\Re}{{\rm Re}\,}
\newcommand{\Si}{{\mathcal S}^0}
\newcommand{\so}{\text{so}}
\newcommand{\spa}{{\rm span}\,}
\newcommand{\Tr}{{\rm Tr}\,}
\newcommand{\Op}{{\rm Op}}
\newcommand{\U}{{\rm U}}
\newcommand{\Wig}{{\mathcal W}}

\newcommand{\CC}{{\mathbb C}}
\newcommand{\HH}{{\mathbb H}}
\newcommand{\RR}{{\mathbb R}}
\newcommand{\TT}{{\mathbb T}}

\newcommand{\Ac}{{\mathcal A}}
\newcommand{\Bc}{{\mathcal B}}
\newcommand{\Cc}{{\mathcal C}}
\newcommand{\Dc}{{\mathcal D}}
\newcommand{\Ec}{{\mathcal E}}
\newcommand{\Fc}{{\mathcal F}}
\newcommand{\Hc}{{\mathcal H}}
\newcommand{\Jc}{{\mathcal J}}
\newcommand{\Lc}{{\mathcal L}}
\renewcommand{\Mc}{{\mathcal M}}
\newcommand{\Nc}{{\mathcal N}}
\newcommand{\Oc}{{\mathcal O}}
\newcommand{\Pc}{{\mathcal P}}
\newcommand{\Sc}{{\mathcal S}}
\newcommand{\Tc}{{\mathcal T}}
\newcommand{\Vc}{{\mathcal V}}
\newcommand{\Uc}{{\mathcal U}}
\newcommand{\Yc}{{\mathcal Y}}

\newcommand{\Bg}{{\mathfrak B}}
\newcommand{\Fg}{{\mathfrak F}}
\newcommand{\Gg}{{\mathfrak G}}
\newcommand{\Ig}{{\mathfrak I}}
\newcommand{\Jg}{{\mathfrak J}}
\newcommand{\Lg}{{\mathfrak L}}
\newcommand{\Pg}{{\mathfrak P}}
\newcommand{\Sg}{{\mathfrak S}}
\newcommand{\Xg}{{\mathfrak X}}
\newcommand{\Yg}{{\mathfrak Y}}
\newcommand{\Zg}{{\mathfrak Z}}

\newcommand{\ag}{{\mathfrak a}}
\newcommand{\bg}{{\mathfrak b}}
\newcommand{\dg}{{\mathfrak d}}
\renewcommand{\gg}{{\mathfrak g}}
\newcommand{\hg}{{\mathfrak h}}
\newcommand{\kg}{{\mathfrak k}}
\newcommand{\mg}{{\mathfrak m}}
\newcommand{\n}{{\mathfrak n}}
\newcommand{\og}{{\mathfrak o}}
\newcommand{\pg}{{\mathfrak p}}
\newcommand{\sg}{{\mathfrak s}}
\newcommand{\tg}{{\mathfrak t}}
\newcommand{\ug}{{\mathfrak u}}
\newcommand{\zg}{{\mathfrak z}}

\newcommand{\ZZ}{\mathbb Z}
\newcommand{\NN}{\mathbb N}
\newcommand{\BB}{\mathbb B}

\newcommand{\ep}{\varepsilon}

\newcommand{\hake}[1]{\langle #1 \rangle }

\newcommand{\scalar}[2]{\langle #1 ,#2 \rangle }
\newcommand{\vect}[2]{(#1_1 ,\ldots ,#1_{#2})}
\newcommand{\norm}[1]{\Vert #1 \Vert }
\newcommand{\normrum}[2]{{\norm {#1}}_{#2}}

\newcommand{\upp}[1]{^{(#1)}}
\newcommand{\p}{\partial}

\newcommand{\opn}{\operatorname}
\newcommand{\slim}{\operatornamewithlimits{s-lim\,}}
\newcommand{\sgn}{\operatorname{sgn}}

\newcommand{\seq}[2]{#1_1 ,\dots ,#1_{#2} }
\newcommand{\loc}{_{\opn{loc}}}

\makeatletter
\title[Modulation spaces for representations of Lie groups]{Modulation spaces of symbols for representations of nilpotent Lie groups}
\author{Ingrid Belti\c t\u a and Daniel Belti\c t\u a}
\address{Institute of Mathematics ``Simion Stoilow'' 
of the Romanian Academy, 
P.O. Box 1-764, Bucharest, Romania}
\email{Ingrid.Beltita@imar.ro}
\email{Daniel.Beltita@imar.ro}
%\translator{}
%\dedicatory{}
%\thanks{\textit{File name}: \texttt{square\_September27\_2009.tex}}
%\thanks{$^{\ast)}$, $^{\ast\ast)}$ Institute of Mathematics ``Simion Stoilow'' 
%of the Romanian Academy, 
%P.O. Box 1-764, Bucharest, Romania}
%\thanks{\textit{Email addresses:} Ingrid.Beltita@imar.ro, Daniel.Beltita@imar.ro}
\date{September 27, 2009}%{\today}
\makeatother

\begin{abstract} 
We investigate continuity properties 
of  operators obtained as values of the Weyl correspondence constructed by 
N.V.~Pedersen (Invent.\ Math.\ 118 (1994), 1--36) 
for arbitrary irreducible representations of nilpotent Lie groups.  
To this end we introduce modulation spaces for such 
representations and establish some of their basic properties.  
The situation of square integrable representations is particularly important 
and in the special case of 
the Schr\"odinger representation of the Heisenberg group 
we recover the classical modulation spaces used in the time-frequency analysis.  
%\textit{2000 MSC:} Primary 47G30; Secondary 22E25,22E65,35S05
%\textit{Keywords:} Weyl calculus; magnetic field; Lie group; semidirect product
\end{abstract}

\maketitle

%\tableofcontents

\section{Introduction}

The representation theory of 
the $(2n+1)$-dimensional Heisenberg group $\HH_{2n+1}$ 
provides a natural background for the pseudo-differential calculus on~$\RR^n$. 
It is well known that the representation theoretic approach 
has led to a deeper understanding of the Weyl calculus, 
which resulted in simplified proofs and improvements 
for many basic results. 
A celebrated example in this connection is the Calder\'on-Vaillancourt theorem 
on $L^2$-boundedness for pseudo-differential operators~(\cite{CV72}). 
This classical theorem was strengthened in the paper~\cite{GH99} 
by using the modulation spaces, which are function (or distribution) spaces 
defined in terms of the Schr\"odinger representations of Heisenberg groups. 
The modulation spaces were introduced in~\cite{Fe83} in the framework of 
harmonic analysis of locally compact abelian groups. 

On the other hand, a remarkable Weyl calculus was set up in \cite{Pe94} 
for arbitrary unitary irreducible representations of any nilpotent Lie group. 
We shall call it the \emph{Weyl-Pedersen calculus}. 
It is a challenging task to understand this interaction 
of the ideas of pseudo-differential calculus with 
the representation theory of nilpotent Lie groups. 

In the present paper we address the above problem in the shape of 
the $L^2$-bound\-ed\-ness theorems. 
Specifically, we are going to investigate continuity properties 
of the operators constructed by 
the Weyl-Pedersen calculus. 
For this purpose we introduce the modulation spaces $M^{r,s}_\phi(\pi)$ 
defined in terms of an arbitrary irreducible representation~$\pi$ 
of a nilpotent Lie group~$G$. 
One key feature of our representation theoretic approach is that 
if $\Oc$ stands for the coadjoint orbit corresponding to $\pi$ (\cite{Ki62}), 
then the symbols of the  operators constructed by the Weyl-Pedersen calculus
are functions or distributions on the coadjoint orbit~$\Oc$, 
while 
the Hilbert space $L^2(\Oc)$ carries a natural irreducible representation 
$\pi^{\#}$  of the nilpotent Lie group $G\ltimes G$. 
Therefore our general notion of modulation spaces for irreducible representations allows us to investigate the modulation spaces of symbols for the operators constructed by the Weyl-Pedersen calculus for  the representation~$\pi$. 
This approach also reveals the representation theoretic background 
of the $L^2$-boundedness theorem of \cite{GH99}. 

We  find several of the familiar properties 
of the classical modulation spaces, such as: 
\begin{itemize}
\item[-] continuity of the  operators constructed by the Weyl-Pedersen calculus with  
symbols in an appropriate modulation space~$M^{\infty,1}_\Phi(\pi^{\#})$ (Corollary~\ref{C3});
\item[-] independence on the choice of a window function, 
and covariance of the Weyl-Pedersen calculus, in the case of square-integrable representations (Theorems \ref{indep_sq} and \ref{cov}).  
\end{itemize}

Besides the aforementioned reasons, 
the present research has also been motivated by the recent interest in 
the magnetic pseudo-differential Weyl calculus on~$\RR^n$ 
(see for instance \cite{MP04}, \cite{IMP07}, \cite{MP09}, 
and the references therein), 
which was partially extended to nilpotent Lie groups in 
the papers \cite{BB09a} and \cite{BB09b}. 
Specifically, the results of the present paper apply to 
the Weyl calculus associated with a polynomial magnetic field on~$\RR^n$, 
in particular complementing the $L^2$-boundedness theorem 
established in \cite{IMP07} for magnetic fields 
whose components are bounded and 
so are also their partial derivatives of arbitrarily high degree.

\subsection*{Notation}
Throughout the paper we denote by $\Sc(\Vc)$ the Schwartz space 
on a finite-dimensional real vector space~$\Vc$. 
That is, $\Sc(\Vc)$ is the set of all smooth functions 
that decay faster than any polynomial together with 
their partial derivatives of arbitrary order. 
Its topological dual ---the space of tempered distributions on $\Vc$--- 
is denoted by $\Sc'(\Vc)$. 

We shall also use the convention that the Lie groups are denoted by 
upper case Latin letters and the Lie algebras are denoted 
by the corresponding lower case Gothic letters. 

For basic notions on Weyl pseudo-differential calculus, we refer to \cite{Hor07}, \cite{Fo89}
and \cite{Gr01}.

\section{Modulation spaces for unitary irreducible representations}

\subsection{Preliminaries on semidirect products}

\begin{definition}\label{sd_def}
\normalfont
Let $G_1$ and $G_2$ be connected Lie groups and assume that we have 
 a continuous group homomorphism 
$\alpha\colon G_1\to\Aut G_2$, $g_1\mapsto\alpha_{g_1}$. 
The corresponding \emph{semidirect product of Lie groups} 
$G_1\ltimes_\alpha G_2$ is the connected Lie group whose 
underlying manifold is the Cartesian product $G_1\times G_2$ 
and whose group operation is given by 
\begin{equation}\label{sd_prod}
(g_1,g_2)\cdot(h_1,h_2)=(g_1h_1,\alpha_{h_1^{-1}}(g_2)h_2) 
\end{equation}
whenever $g_j,h_j\in G_j$ for $j=1,2$. 

Let us denote by $\dot\alpha\colon\gg_1\to\Der\gg_2$ 
the homomorphism of Lie algebras defined as the differential 
of the Lie group homomorphism $G_1\to\Aut\gg_2$, $g_1\mapsto\Lie(\alpha_{g_1})$. 
Then the \emph{semidirect product of Lie algebras} 
$\gg_1\ltimes_{\dot\alpha}\gg_2$ 
is the Lie algebra whose underlying linear space is the Cartesian product 
$\gg_1\times\gg_2$ with the Lie bracket given by 
\begin{equation}\label{sd_bracket}
[(X_1,X_2),(Y_1,Y_2)]=([X_1,Y_1],\dot\alpha(X_1)Y_2-\dot\alpha(Y_1)X_2+[X_2,Y_2])
\end{equation}
if $X_j,Y_j\in\gg_j$ for $j=1,2$. 
One can prove that $\gg_1\ltimes_{\dot\alpha}\gg_2$ is the Lie algebra 
of the Lie group $G_1\ltimes_\alpha G_2$ 
(see for instance Ch.~9 in \cite{Ho65}). 
\qed
\end{definition}

\begin{remark}\label{sd_rem}
\normalfont
Let $G_1$ and $G_2$ be nilpotent Lie groups and  
$\alpha\colon G_1\to\Aut G_2$ be a \emph{unipotent automorphism}, 
in the sense that for every $X_1\in\gg_1$ there exists an integer $m\ge1$ such that 
$\dot\alpha(X_1)^m=0$. 
Then an inspection of \eqref{sd_bracket} shows that $\gg_1\ltimes_{\dot\alpha}\gg_2$ 
is a nilpotent Lie algebra, hence $G_1\ltimes_\alpha G_2$ is a nilpotent Lie group. 
\qed
\end{remark}

\begin{example}\label{sd_ex2}
\normalfont
For an arbitrary Lie group $G$ with the center $Z$, 
let us specialize Definition~\ref{sd_def} for $G_1=G_2:=G$ and 
$\alpha\colon G\to\Aut G$, $g\mapsto\alpha_g$, 
where $\alpha_g(h)=ghg^{-1}$ whenever $g,h\in G$. 
Then the corresponding semidirect product will always be denoted simply by $G\ltimes G$ 
and has the following properties: 
\begin{enumerate}
\item\label{sd_ex2_item1} 
If $G$ is nilpotent, then so is $G\ltimes G$. 
\item\label{sd_ex2_item2}  
The Lie algebra of $G\ltimes G$ is $\gg\ltimes_{\ad_{\gg}}\gg$,  
which will be denoted simply by $\gg\ltimes\gg$,  
and the center of $G\ltimes G$ is $Z\times Z$. 
\item\label{sd_ex2_item3} 
The exponential map of the Lie group $G\ltimes G$ is given by 
$$\exp_{G\ltimes G}(X,Y)=(\exp_G X,\exp_G(-X)\exp_G(X+Y)) $$
for every $(X,Y)\in\gg\ltimes_{\ad_{\gg}}\gg$. 
\item\label{sd_ex2_item4} The mapping 
$$\mu\colon G\times G\to G\ltimes G,\quad (g,h)\mapsto(gh,g) $$
is an isomorphism of Lie groups, and the corresponding isomorphism of Lie algebras is 
$\Lie(\mu)\colon\gg\times\gg\to\gg\ltimes\gg$, $(X,Y)\mapsto(X+Y,X)$. 
\end{enumerate}
In fact, property~\eqref{sd_ex2_item1} follows by 
Remark~\ref{sd_rem}. 
Property~\eqref{sd_ex2_item2} is a consequence of the fact that 
$\dot\alpha=\ad_{\gg}\colon\gg\to\Der\gg$ along with~\eqref{sd_prod}. 

To prove property~\eqref{sd_ex2_item3}, note that the mapping 
$\Pi\colon G\ltimes G\to G$, $(g_1,g_2)\to g_1g_2$ is a homomorphism of Lie groups, 
hence we have the commutative diagram 
$$\begin{CD}
\gg\ltimes\gg @>{\Lie(\Pi)}>> \gg \\
@V{\exp_{G\ltimes G}}VV @VV{\exp_G}V \\
G\ltimes G @>>{\Pi}> G
\end{CD}
$$
where it is easy to see that the Lie algebra homomorphism 
$\Lie(\Pi)\colon \gg\ltimes\gg\to\gg$ is given by $(X,Y)\mapsto X+Y$. 
Now let $(X,Y)\in\gg\ltimes\gg$ arbitrary. 
It is clear that there exists $g\in G$ such that $\exp_{G\ltimes G}(X,Y)=(\exp_G X,g)$, 
and then the above commutative diagram shows that 
$\exp_G(X+Y)=\Pi(\exp_{G\ltimes G}(X,Y))=\Pi(\exp_G X,g)=(\exp_G X)g$, 
whence $g=\exp_G(-X)\exp_G(X+Y)$. 

Finally, property~\eqref{sd_ex2_item4} follows by a straightforward computation. 
\qed
\end{example}

\subsection{Weyl-Pedersen calculus for unitary irreducible representations}

\begin{setting}\label{predual_sett}
\normalfont 
Throughout the present section we shall use the following notation:
\begin{enumerate}
 \item Let $G$ be a connected, simply connected, nilpotent Lie group with the Lie algebra~$\gg$.  
 Then the exponential map $\exp_G\colon\gg\to G$ is a diffeomorphism 
 with the inverse denoted by $\log_G\colon G\to\gg$. 
 \item We denote by $\gg^*$ the linear dual space to $\gg$ and 
  by $\hake{\cdot,\cdot}\colon\gg^*\times\gg\to\RR$ 
  the natural duality pairing. 
 \item Let $\xi_0\in\gg^*$ with the corresponding coadjoint orbit $\Oc:=\Ad_G^*(G)\xi_0\subseteq\gg^*$.  
 \item The \emph{isotropy group} at $\xi_0$ is $G_{\xi_0}:=\{g\in G\mid\Ad_G^*(g)\xi_0=\xi_0\}$ 
 with the corresponding \emph{isotropy Lie algebra} $\gg_{\xi_0}=\{X\in\gg\mid\xi_0\circ\ad_{\gg}X=0\}$. 
 The \emph{center} $\zg:=\{X\in\gg\mid[X,\gg]=\{0\}\}$ clearly satisfies $\zg\subseteq\gg_{\xi_0}$. 
 \item Let $n:=\dim\gg$ and fix a sequence of ideals in $\gg$, 
$$\{0\}=\gg_0\subset\gg_1\subset\cdots\subset\gg_n=\gg$$
such that $\dim(\gg_j/\gg_{j-1})=1$ and $[\gg,\gg_j]\subseteq\gg_{j-1}$ 
for $j=1,\dots,n$. 
 \item Pick any $X_j\in\gg_j\setminus\gg_{j-1}$ for $j=1,\dots,n$, 
so that the set $\{X_1,\dots,X_n\}$ will be a \emph{Jordan-H\"older basis} in~$\gg$. 
\item The set of \emph{jump indices} of the coadjoint orbit $\Oc$ 
with respect to the above Jordan-H\"older basis is  
$e:=\{j\in\{1,\dots,n\}\mid \gg_j\not\subseteq\gg_{j-1}+\gg_{\xi_0}\}$ 
and does not depend on the choice of $\xi_0\in\Oc$ 
(see also Prop.~2.4.1 in \cite{Pe84}). 
The corresponding \emph{predual of the coadjoint orbit}~$\Oc$ is  
$$\gg_e:=\spa\{X_j\mid j\in J_{\xi_0}\}\subseteq\gg.$$
We shall denote $e=\{j_1,\dots,j_d\}$ with $1\le j_1<\cdots<j_d\le n$. 
\item We shall always consider $\Oc$ endowed with 
its canonical Liouville measure 
(see for instance the remark after the statement of 
the theorem in \S~6, Ch. II, Part 2 in \cite{Pu67}).
 \item Let $\pi\colon G\to\Bc(\Hc)$ be a fixed unitary irreducible representation 
associated with the coadjoint orbit $\Oc$ by Kirillov's theorem (\cite{Ki62}). 
\end{enumerate}
\qed 
\end{setting}

\begin{remark}\label{distrib}
\normalfont
The space of \emph{smooth vectors}  
$\Hc_\infty:=\{v\in\Hc\mid \pi(\cdot)v\in\Ci(G,\Hc)\}$ 
is a Fr\'echet space in a natural way and is a dense linear subspace of $\Hc$ 
which is invariant under the unitary operator $\pi(g)$ for every $g\in G$.  
The \emph{derivate representation} $\de\pi\colon\gg\to\End(\Hc_\infty)$  
is a homomorphism of Lie algebras defined by 
$$(\forall X\in\gg,v\in\Hc_\infty)\quad\de\pi(X)v=\frac{\de}{\de t}\Big{\vert}_{t=0}\pi(\exp_G(tX))v. $$
We denote by $\Hc_{-\infty}$ the space of all continuous antilinear functionals on $\Hc_\infty$ 
and the corresponding pairing will be denoted by 
$(\cdot\mid\cdot)\colon\Hc_{-\infty}\times\Hc_\infty\to\CC$.  
just as the scalar product in $\Hc$, since they agree on $\Hc_\infty\times\Hc_{\infty}$ 
if we think of the natural inclusions $\Hc_\infty\hookrightarrow\Hc\hookrightarrow\Hc_{-\infty}$. 
(See for instance \cite{Ca76} for more details.)
\qed
\end{remark}

\begin{remark}\label{smoothness}
\normalfont 
We now recall a few facts from subsect.~1.2 in \cite{Pe94} for later use.
Let us denote by $\Sg_p(\Hc)$ the Schatten ideals of operators on $\Hc$ for $1\le p\le\infty$. 
Consider the unitary representation 
$\pi^{\otimes 2}\colon G\times G\to\Bc(\Sg_2(\Hc))$ defined by 
$$(\forall g_1,g_2\in G)(\forall T\in\Sg_2(\Hc))\quad 
\pi^{\otimes 2}(g_1,g_2)T=\pi(g_1)T\pi(g_2)^{-1}.$$
It is not difficult to see that $\pi^{\otimes 2}$ is strongly continuous.  
The corresponding space of smooth vectors is denoted by $\Bc(\Hc)_\infty$ 
and is called the space of \emph{smooth operators} for the representation~$\pi$. 
One can prove that actually $\Bc(\Hc)_\infty\subseteq\Sg_1(\Hc)$. 

For an alternative description of $\Bc(\Hc)_\infty$ 
let $\gg_{\CC}:=\gg\otimes_{{\RR}}\CC$ be the complexification of $\gg$ 
with the corresponding universal associative enveloping algebra $\U(\gg_{\CC})$. 
Then the aforementioned homomorphism of Lie algebras $\de\pi$ has a unique extension 
to a homomorphism of unital associative algebras $\de\pi\colon\U(\gg_{\CC})\to\End(\Hc_\infty)$. 
One can prove that for $T\in\Bc(\Hc)$ we have $T\in\Bc(\Hc)_\infty$ if and only if  $T(\Hc)+T^*(\Hc)\subseteq\Hc_\infty$ and $\de\pi(u)T,\de\pi(u)T^*\in\Bc(\Hc)$ 
for every $u\in\U(\gg_{\CC})$. 

Since 
$\{(\cdot\mid f_1)f_2\mid f_1,f_2\in\Hc_\infty\}
\subseteq\Bc(\Hc)_\infty\subseteq\Sg_1(\Hc)$ and $\Hc_\infty$ 
is dense in $\Hc$, we get continuous inclusion maps 
\begin{equation}\label{smoothness_eq1}
\Bc(\Hc)_\infty\hookrightarrow\Sg_1(\Hc)\hookrightarrow\Bc(\Hc)
\hookrightarrow\Bc(\Hc)_\infty^*,
\end{equation}
where the latter mapping is constructed by using 
the well-known isomorphism $(\Sg_1(\Hc))^*\simeq\Bc(\Hc)$ 
given by the usual semifinite trace on $\Bc(\Hc)$.
\qed
\end{remark}

\begin{definition}\label{amb}
\normalfont
The Fourier transform 
$\Sc(\Oc)\to\Sc(\gg_e)$, $a\mapsto\widehat{a}$,
defined by 
$$\widehat{a}(X)=\int\limits_{\Oc}\ee^{-\ie\hake{\xi,X}}a(\xi)\de\xi$$
is an isomorphism of Fr\'echet spaces.
The Lebesgue measure on $\gg_e$ can be normalized 
such that the Fourier transform extends to a unitary operator 
$$L^2(\Oc)\to L^2(\gg_e),\quad a\mapsto\widehat{a},$$
and its inverse is defined by the usual formula 
(see Lemma~4.1.1. in \cite{Pe94}). 
We shall always consider the predual $\gg_e$ endowed with 
this normalized measure. 

If $f\in\Hc_{-\infty}$ and $\phi\in\Hc_\infty$, or $f,\phi\in\Hc$, 
then we define the corresponding \emph{ambiguity function} 
$$\Ac(f,\phi)=\Ac_\phi f\colon\gg_e\to\CC,\quad (\Ac_\phi f)(X)=(f\mid\pi(\exp_G X)\phi).$$ 
For $\phi\in\Hc_{-\infty}$ and $f\in\Hc_\infty$ we also define 
$(\Ac_\phi f)(X)=\overline{(\phi\mid\pi(\exp_G(-X))f)}$ whenever $X\in\gg_e$. 

It follows by Proposition~\ref{orthog}\eqref{orthog_item1} below that if $f,\phi\in\Hc$, 
then $\Ac_\phi f\in L^2(\gg_e)$, 
so we can use the aforementioned Fourier transform to define 
the corresponding \emph{cross-Wigner distribution} $\Wig(f,\phi)\in L^2(\Oc)$ 
such that 
$\widehat{\Wig(f,\phi)}:=\Ac_\phi f$.
\qed
\end{definition}

The second equality in Proposition~\ref{orthog}\eqref{orthog_item1} below 
could be referred to as the \emph{Moyal identity} since  
that classical identity (see for instance \cite{Gr01}) 
is recovered 
in the special case when $G$ is a simply connected Heisenberg group.

\begin{proposition}\label{orthog}
 The following assertions hold: 
\begin{enumerate}
 \item\label{orthog_item1} 
If $\phi\in\Hc$, then $\Ac_\phi f\in L^2(\gg_e)$. 
We have 
\begin{equation}\label{orthog_eq1}
\begin{aligned}
({\Ac}_{\phi_1}f_1\mid {\Ac}_{\phi_2}f_2)_{L^2(\gg_e)} 
&=(f_1\mid f_2)_{\Hc}\cdot(\phi_2\mid\phi_1)_{\Hc} \\
&=({\Wig}(f_1,\phi_1)\mid {\Wig}(f_2,\phi_2))_{L^2(\Oc)} 
\end{aligned}
\end{equation}
for arbitrary $\phi_1,\phi_2,f_1,f_2\in\Hc$. 
\item\label{orthog_item2} 
If $\phi_0\in\Hc$ with $\Vert\phi_0\Vert=1$, 
then the operator ${\Ac}_{\phi_0}\colon\Hc\to L^2(\gg_e)$, $f\mapsto {\Ac}_{\phi_0} f$, 
is an isometry and we have 
\begin{equation*}
\int\limits_{\gg_e}({\Ac}_{\phi_0}f)(X)\cdot\pi(\exp_G X)\phi\,\de X
=(\phi\mid\phi_0)f
\end{equation*}
for every $\phi\in\Hc_\infty$ and $f\in\Hc$. 
In particular, 
\begin{equation*}
\int\limits_{\gg_e}({\Ac}_{\phi_0} f)(X)\cdot\pi(\exp_G X)\phi_0\,\de X=f
\end{equation*}
for arbitrary $f\in\Hc$. 
\end{enumerate}

\end{proposition}

\begin{proof} 
\eqref{orthog_item1}
We first prove that \eqref{orthog_eq1} holds for $\phi_1,\phi_2,f_1,f_2\in\Hc_\infty$. 
Since  $\Bc(\Hc_\infty)$ is contained in the ideal $\Sg_1(\Hc)$ of trace-class operators, 
it makes sense to define 
$$(\forall A\in\Bc(\Hc)_\infty)\quad f_\pi^A\colon G\to\CC,\quad f_\pi^A(x)=\Tr(\pi(x)A).$$
It follows by Th.~2.2.7 in \cite{Pe94} that for the suitably normalized 
Lebesgue measure on $\gg_e$ we have for every $A,B\in\Bc(\Hc)_\infty$,  
\begin{equation}\label{orthog_proof_eq2} 
\int\limits_{\gg_e} f_\pi^A(\exp_G X)\overline{f_\pi^B(\exp_G X)}\de X  =  \Tr(AB^*). 
\end{equation}
We now denote 
$$(\forall f,\phi\in\Hc)\quad A_{f,\phi}=(\cdot\mid\phi)f\in\Bc(\Hc) $$
and recall that for arbitrary $f,f_1,f_2,\phi,\phi_1,\phi_2\in\Hc$ we have
$$A_{f,\phi}^*=A_{\phi,f},\ 
\Tr(A_{f,\phi})=(f\mid\phi),\ \text{and}\ 
A_{f_1,\phi_1}A_{f_2,\phi_2}=A_{f_1,(\phi_1\mid f_2)\phi_2}
=A_{(f_2\mid \phi_1)f_1,\phi_2}. $$
It then easily follows that if $f,\phi\in\Hc_\infty$, 
then $A_{f,\phi}\in\Bc(\Hc)_\infty$ and 
for arbitrary $X\in\gg_e$ we have 
$$\begin{aligned}
f_\pi^{A_{f,\phi}}(\exp_G X)
&=\Tr(\pi(\exp_G X)A_{f,\phi})
=\Tr(A_{\pi(\exp_G X)f,\phi})
=(\pi(\exp_G X)f\mid\phi) \\
&=(f\mid\pi(\exp_G(-X))\phi),
\end{aligned}$$
whence 
\begin{equation}\label{orthog_proof_eq3}
(\forall X\in\gg_e)\quad f_\pi^{A_{f,\phi}}(\exp_G X)=(\Ac_\phi f)(-X). 
\end{equation}
Now, by using \eqref{orthog_proof_eq2} for $A:=A_{f_1,\phi_1}$ and $B:=A_{f_2,\phi_2}$, 
we get 
$$\begin{aligned}
({\Ac}_{\phi_1}f_1\mid {\Ac}_{\phi_2}f_2)_{L^2(\gg_e)} 
&=\Tr(A_{f_1,\phi_1}A_{f_2,\phi_2}^*) 
=\Tr(A_{f_1,\phi_1}A_{\phi_2,f_2}) 
=\Tr(A_{f_1,(\phi_1\mid\phi_2)f_2}) \\
&=(f_1\mid(\phi_1\mid\phi_2)f_2) 
=(f_1\mid f_2)_{\Hc}\cdot(\phi_2\mid\phi_1)_{\Hc} 
\end{aligned}$$
The second part of \eqref{orthog_eq1} then follows since 
the Fourier transform $L^2(\Oc)\to L^2(\gg_e)$ is a unitary operator, 
as we already mentioned in Definition~\ref{amb}. 

The extension of \eqref{orthog_eq1} from $\Hc_\infty$ to $\Hc$ 
proceeds by a density argument. 
First note that by \eqref{orthog_eq1} for $\phi_1=\phi_2=:\phi\in\Hc_\infty$ 
and $f_1=f_2=:f\in\Hc_\infty$ 
we get $\Vert\Ac_\phi f\Vert=\Vert\phi\Vert\cdot\Vert f\Vert$. 
Since $\Hc_\infty$ is dense in $\Hc$, it then follows that 
the sesquilinear mapping $\Hc_\infty\times\Hc_\infty\to\Hc$, $(f,\phi)\mapsto\Ac_\phi f$
extends uniquely to a sesquilinear mapping $\Hc\times\Hc\to\Hc$ satisfying 
\begin{equation}\label{orthog_proof_eq4}
(\forall f,\phi\in\Hc)\quad  \Vert\Ac_\phi f\Vert=\Vert\phi\Vert\cdot\Vert f\Vert.
\end{equation}
Now the first part of \eqref{orthog_eq1} follows as a polarization of \eqref{orthog_proof_eq4}, 
and then the second part follows by using 
the Fourier transform $L^2(\Oc)\to L^2(\gg_e)$ as above. 

\eqref{orthog_item2} 
It follows at once by Assertion~\eqref{orthog_item1} that 
the operator ${\Ac}_{\phi_0}\colon\Hc\to L^2(\gg_e)$ 
is an isometry if $\Vert\phi_0\Vert=1$. 
The other properties then follow immediately; 
see for instance Proposition~2.11 in \cite{Fue05}. 
\end{proof}

We now draw some useful consequences of Proposition~\ref{orthog}. 
We emphasize that 
Assertion~\eqref{orthog_distrib_item3} in the following corollary  
in the special case of square-integrable representations 
reduces to a theorem of \cite{Co84} and \cite{CM96}. 
One thus recovers Th.~2.3 in \cite{GZ01} 
in the case of the Schr\"odinger representation 
of the Heisenberg group. 

\begin{corollary}\label{orthog_distrib}
If $\phi_0\in\Hc_\infty$ with $\Vert\phi_0\Vert=1$, 
then the following assertions hold: 
\begin{enumerate}
\item\label{orthog_distrib_item1} 
For every $f\in\Hc_{-\infty}$ we have 
\begin{equation}\label{orthog_distrib_eq1}
\int\limits_{\gg_e}({\Ac}_{\phi_0} f)(X)\cdot\pi(\exp_G X)\phi_0\,\de X=f
\end{equation} 
where the integral is convergent in the weak$^*$-topology of $\Hc_{-\infty}$. 
\item\label{orthog_distrib_item2} 
If $f\in\Hc_\infty$, then the above integral converges in 
the Fr\'echet topology of $\Hc_\infty$. 
\item\label{orthog_distrib_item3} 
If $f\in\Hc_{-\infty}$, then we have 
$f\in\Hc_\infty$ if and only if $\Ac_{\phi_0}f\in\Sc(\gg_e)$. 
\end{enumerate}
\end{corollary}

\begin{proof}
If $f\in\Hc_{-\infty}$, we have to prove that  
$\int\limits_{\gg_e}({\Ac}_{\phi_0} f)(X)\cdot(\pi(\exp_G X)\phi_0\mid\phi)\,\de X=(f\mid\phi)$, 
for every $\phi\in\Hc_\infty$, 
that is, 
\begin{equation*} 
\int\limits_{\gg_e}(f\mid\pi(\exp_G X)\phi_0)\cdot(\pi(\exp_G X)\phi_0\mid\phi)\,\de X=(f\mid\phi). 
\end{equation*}
Since $(f\mid\cdot)$ is an antilinear continuous functional, 
the above equation will follow as soon as we have proved 
that for $\phi\in\Hc_\infty$ we have 
$$\int\limits_{\gg_e}(\phi\mid\pi(\exp_G X)\phi_0)\pi(\exp_G X)\phi_0\,\de X=\phi$$
with an integral that converges in the topology of $\Hc_\infty$. 
Note that this is precisely Assertion~\eqref{orthog_distrib_item2}. 
To prove it, we just have to use 
Proposition~\ref{orthog}\eqref{orthog_item2} along with the fact that 
for $\phi,\phi_0\in\Hc_\infty$ 
the function $X\mapsto(\phi\mid\pi(\exp_G X)\phi_0)=(\Ac_{\phi_0}\phi)(X)$ 
belongs to $\Sc(\gg_e)$ (see Th.~2.2.6 in \cite{Pe94}) 
while the function $X\mapsto\pi(\exp_G X)\phi_0$ 
and all its partial derivatives have polynomial growth. 

For Assertion~\eqref{orthog_distrib_item3}, 
we have just noted that if $f\in\Hc_\infty$ then 
$\Ac_{\phi_0}f\in\Sc(\gg_e)$ as a direct consequence of 
Th.~2.2.6 in \cite{Pe94}. 
Conversely, if $f\in\Hc_{-\infty}$ has the property 
$\Ac_{\phi_0}f\in\Sc(\gg_e)$, then the fact that all 
the partial derivatives of $X\mapsto\pi(\exp_G X)\phi_0$ have polynomial growth 
implies at once that the integral in \eqref{orthog_distrib_eq1} 
is convergent in the Fr\'echet space $\Hc_\infty$, 
hence Assertion~\eqref{orthog_distrib_item1} shows that 
actually $f\in\Hc_\infty$.
\end{proof}

\begin{definition}\label{calc_def}
\normalfont 
The \emph{Weyl-Pedersen calculus} $\Op^\pi(\cdot)$ for the unitary representation~$\pi$ is defined 
for every $a\in\Sc(\Oc)$ by 
\begin{equation}\label{calc_formula}
\Op^\pi(a)=\int\limits_{\gg_e}\widehat{a}(X)\pi(\exp_GX)\de X\in\Bc(\Hc). 
\end{equation}
This definition can be extended to an arbitrary tempered distribution $a\in\Sc'(\Oc)$ by using  
Th.~4.1.4 and Th.~2.2.7 in \cite{Pe94}
to define an unbounded operator 
$\Op^\pi(a)$ such that  
\begin{equation}\label{duality}
(\forall b\in\Sc(\Oc))\quad \Tr(\Op^\pi(a)\Op^\pi(b))=\hake{a,b}, 
\end{equation}
where we recall that $\hake{\cdot,\cdot}\colon\Sc'(\Oc)\times\Sc(\Oc)\to\CC$ 
stands for the usual pairing between 
the tempered distributions and the Schwartz functions. 
We say that $a\in\Sc'(\Oc)$ is 
the \emph{symbol} of the operator $\Op^\pi(a)$.
\qed 
\end{definition}

We now record some basic properties 
of the Weyl-Pedersen calculus constructed in Definition~\ref{calc_def}. 
These are actually direct consequences of Proposition~\ref{orthog}\eqref{orthog_item1}. 

\begin{corollary}\label{pseudo_prop}
The following assertions hold: 
\begin{enumerate}
\item\label{pseudo_prop_item2} 
For each $a\in\Sc(\Oc)$ we have 
$$
(\Op^\pi(a)\phi\mid f)_{\Hc}=(\widehat{a}\mid\Ac_\phi f)_{L^2(\gg_e)}
=(a\mid\Wig(f,\phi))_{L^2(\Oc)}$$
whenever $\phi,f\in\Hc$. 
Similar equalities hold if $a\in\Sc'(\Oc)$ and $\phi,f\in\Hc_\infty$. 
\item\label{pseudo_prop_item3} 
If $\phi_1,\phi_2\in\Hc_\infty$ and $a:=\Wig(\phi_1,\phi_2)\in\Sc(\Oc)$, 
then $\Op^\pi(a)$ is a rank-one operator, namely 
$\Op^\pi(a)=(\cdot \mid\phi_2) \phi_1$.  
\end{enumerate}
\end{corollary}

\begin{proof}
Assertion~\eqref{pseudo_prop_item2} is a consequence of formula~\eqref{calc_formula} 
along with Definition~\ref{amb}. 
Then Assertion~\eqref{pseudo_prop_item3} follows by Assertion~\eqref{pseudo_prop_item2} 
along with Proposition~\ref{orthog}\eqref{orthog_item1}. 
In fact, 
$$\begin{aligned}
(\Op^\pi(\Wig(\phi_1,\phi_2))f\mid\phi)
&=(\Wig(\phi_1,\phi_2)\mid\Wig(\phi,f))
=(\phi_1\mid\phi)\cdot(f\mid\phi_2) \\
&=((f\mid\phi_2)\phi_1 \mid\phi) 
\end{aligned}$$
for arbitrary $\phi\in\Hc$. 
\end{proof}

\begin{remark}\label{wig_ext}
\normalfont
We can define the cross-Wigner distribution $\Wig(f_1,f_2)\in\Sc'(\Oc)$ 
for arbitrary $f_1,f_2\in\Hc_{-\infty}$ as follows. 
An application of Th.~1.3(b) in \cite{Ca76} shows 
that if $A\in\Bc(\Hc)_\infty$  and $f\in\Hc_{-\infty}$,  
then $Af\in\Hc_\infty$, in the sense that there exists a smooth vector 
denoted $Af$ such that for every $\phi\in\Hc_\infty$ we have 
$(f\mid A^*\phi)=(Af\mid\phi)$. 
Moreover, we thus get a continuous linear map $A\colon\Hc_{-\infty}\to\Hc_\infty$ whose restriction to $\Hc$ 
is the original operator $A\in\Bc(\Hc)_\infty$.
Then for $f_1,f_2\in\Hc_{-\infty}$ we can define 
the continuous antilinear functional
$$T_{f_1,f_2}\colon\Bc(\Hc)_\infty\to\CC,\quad 
T_{f_1,f_2}(A):=(f_1\mid Af_2).$$
That is, $T_{f_1,f_2}\in\Bc(\Hc)_\infty^*$, 
and then Th.~4.1.4(5) in \cite{Pe94} shows that 
there exists a unique distribution $a_{f_1,f_2}\in\Sc'(\Oc)$ 
such that $\Op^\pi(a_{f_1,f_2})=T_{f_1,f_2}$. 
Now define 
$$\Wig(f_1,f_2):=a_{f_1,f_2}.$$ 
We can consider the rank-one operator 
$S_{f_1,f_2}:=(\cdot\mid f_2)f_1\colon \Hc_\infty\to\Hc_{-\infty}$ and 
for arbitrary $A\in\Bc(\Hc)_\infty$ thought of as 
a continuous linear map $A\colon\Hc_{-\infty}\to\Hc_\infty$ as above 
we have 
$$\Tr(S_{f_1,f_2}A)=(f_1\mid Af_2)=T_{f_1,f_2}(A).$$ 
Thus the trace duality pairing allows us to identify 
the functional $T_{f_1,f_2}\in\Bc(\Hc)_\infty^*$ 
with the rank-one operator $(\cdot\mid f_2)f_1$, 
and then we can write 
\begin{equation}\label{wig_ext_eq1}
(\forall f_1,f_2\in\Hc_{-\infty})\quad 
\Op^\pi(\Wig(f_1,f_2))=(\cdot\mid f_2)f_1.
\end{equation}
In particular, it follows that the above extension of 
the cross-Wigner distribution to a mapping 
$\Wig(\cdot,\cdot)\colon\Hc_{-\infty}\times\Hc_{-\infty}\to\Sc'(\Oc)$ 
allows us to generalize 
the assertion of Corollary~\ref{pseudo_prop}\eqref{pseudo_prop_item3} 
to arbitrary $\phi_1,\phi_2\in\Hc_{-\infty}$. 
\qed
\end{remark}

\begin{definition}\label{moyal}
\normalfont
Recall from Th.~4.1.4(5) in \cite{Pe94} that the Weyl-Pedersen calculus 
$\Op^\pi\colon\Sc'(\Oc)\to\Bc(\Hc)_\infty^*$ 
is a linear isomorphism and a weak$^*$-homeomorphism. 
We introduce the linear space 
$$\Si(\Oc):=\{a\in\Sc'(\Oc)\mid \Op^\pi(a)\in\Bc(\Hc)\} $$
(see \eqref{smoothness_eq1}). 
Then the mapping $\Op^\pi$ induces a linear isomorphism 
$\Si(\Oc)\to\Bc(\Hc)$, hence there exists an uniquely 
defined bilinear associative \emph{Moyal product} 
$$\Si(\Oc)\times\Si(\Oc)\to\Si(\Oc), \quad (a,b)\mapsto a\# b $$
such that 
$$(\forall a,b\in\Si(\Oc))\quad \Op^\pi(a\# b)=\Op^\pi(a)\Op^\pi(b). $$
The space of distributions $\Si(\Oc)$ is thus made into 
a $W^*$-algebra such that 
the mapping $\Si(\Oc)\to\Bc(\Hc)$, $a\mapsto\Op^\pi(a)$  
is a $*$-isomorphism. 
\qed
\end{definition}

With Definition~\ref{moyal} at hand, we can say that 
one of the main problems addressed in the present paper 
is to describe large classes of distributions 
belonging in the space~$\Si(\Oc)$. 

\begin{example}\label{exs}
\normalfont
Here are some examples of distributions in $\Si(\Oc)$
which are already available.
\begin{enumerate}
\item\label{exs_item1} 
It follows at once by \eqref{calc_formula} and \eqref{duality} 
that 
$$\{a\in\Sc'(\Oc)\mid \widehat{a}\in L^1(\gg_e)\}\subseteq\Si(\Oc).$$ 
\item\label{exs_item2}  
The Schwartz space $\Sc(\Oc)$ is a $*$-subalgebra of $\Si(\Oc)$ 
and the mapping $\Op^\pi\colon\Sc(\Oc)\to\Bc(\Hc)_\infty$ 
is an algebra $*$-isomorphism by Th.~4.1.4 in~\cite{Pe94}. 
\item\label{exs_item3}  
The space $L^2(\Oc)$ is a $*$-subalgebra of $\Si(\Oc)$,  
and $\Op^\pi\colon L^2(\Oc)\to\Sg_2(\Hc)$ 
is a unitary operator and an algebra $*$-isomorphism 
as an easy consequence of Th.~4.1.4 in~\cite{Pe94}; 
see also \cite{Ma07}. 
\item\label{exs_item4}  
For every $Y\in\gg$ we have $\ee^{\ie\hake{\cdot,Y}}\in\Si(\Oc)$ 
since it follows at once by \eqref{calc_formula} and \eqref{duality} 
that $\Op^\pi(\ee^{\ie\hake{\cdot,Y}})=\pi(\exp_G Y)$. 
\end{enumerate}
See also Corollary~\ref{C3} for 
the important example~$M^{\infty,1}_\Phi(\pi^\#)\hookrightarrow \Si(\Oc)$. 
\qed
\end{example}

\subsection{Modulation spaces}

\begin{definition}\label{modular_def}
\normalfont
Let $\phi\in\Hc_\infty\setminus\{0\}$ be fixed and 
assume that we have a direct sum decomposition 
$\gg_e=\gg_e^1\dotplus\gg_e^2$.

Then let $1\le r,s\le\infty$ and   
for arbitrary $f\in\Hc_{-\infty}$ define 
$$\Vert f\Vert_{M^{r,s}_\phi}
=\Bigl(\int\limits_{\gg_e^2}
\Bigl(\int\limits_{\gg_e^1}
\vert(\Ac_\phi f)(X_1,X_2)\vert^r\de X_1 \Bigr)^{s/r}
\de X_2\Bigr)^{1/s}\in[0,\infty] $$
with the usual conventions if $r$ or $s$ is infinite. 
Then we call the space 
$$M^{r,s}_\phi(\pi):=\{f\in\Hc_{-\infty}\mid\Vert f\Vert_{M^{r,s}_\phi}<\infty\}$$ 
a \emph{modulation space} for the irreducible unitary representation $\pi\colon G\to\Bc(\Hc)$ 
with respect to the decomposition $\gg_e\simeq\gg_e^1\times\gg_e^2$ 
and the \emph{window vector} $\phi\in\Hc_\infty\setminus\{0\}$. 
\qed
\end{definition}

\begin{remark}\label{modular_mixed}
\normalfont
Assume the setting of Definition~\ref{modular_def} and recall  
the \emph{mixed-norm space} $L^{r,s}(\gg_e^1\times\gg_e^2)$ 
consisting of the (equivalence classes of) 
Lebesgue measurable functions $\Theta\colon \gg_e^1\times\gg_e^2\to\CC$ 
such that 
$$\Vert\Theta\Vert_{L^{r,s}}
:=\Bigl(\int\limits_{\gg_e^2}
\Bigl(\int\limits_{\gg_e^1}
\vert(\Theta(X_1,X_2)\vert^r\de X_1 \Bigr)^{s/r}
\de X_2\Bigr)^{1/s}<\infty $$
(cf.~\cite{Gr01}). 
It is clear that 
$M^{r,s}_\phi(\pi)=\{f\in\Hc_{-\infty}\mid \Ac_\phi f\in L^{r,s}(\gg_e^1\times\gg_e^2)\}$. 
\qed
\end{remark}

\begin{example}\label{mod_L2}
\normalfont
For any choice of $\phi\in\Hc_\infty\setminus\{0\}$ in Definition~\ref{modular_def} we have 
$$M^{2,2}_\phi(\pi)=\Hc.$$
Indeed, this equality holds since 
$\Vert\Ac_\phi f\Vert_{L^2(\gg_e)}=\Vert\phi\Vert\cdot\Vert f\Vert$ 
for every $f\in\Hc$ 
(see \eqref{orthog_proof_eq4} in the proof of Proposition~\ref{orthog} above). 
\qed
\end{example}

\subsection{Continuity of Weyl-Pedersen calculus on modulation spaces}

In the following lemma we use notation introduced in Example~\ref{sd_ex2}\eqref{sd_ex2_item4} and Remark~\ref{smoothness}. 

\begin{lemma}\label{pi_bar}
Let $G$ be any Lie group with a unitary irreducible representation 
$\pi\colon G\to\Bc(\Hc)$  and define 
$$\bar\pi\colon G\ltimes G\to\Bc(\Sg_2(\Hc)),\quad \bar\pi(g,h)T=\pi(gh)T\pi(g)^{-1}.$$
Then the following assertions hold: 
\begin{enumerate}
 \item\label{pi_bar_item1} 
The diagram 
$$\xymatrix{
 G\ltimes G\ar[r]^{\bar\pi}  & \Bc(\Sg_2(\Hc))   \\
  G\times G \ar[u]^{\mu} \ar[ur]_{\pi^{\otimes 2}}                 }
$$
is commutative and 
$\bar\pi$ is a unitary irreducible representation of $G\ltimes G$. 
\item\label{pi_bar_item3} 
The space of smooth vectors for the representation $\bar\pi$ 
is $\Bc(\Hc)_\infty$.
\item\label{pi_bar_item4} 
Let us denote by $\bar\gg=\gg\ltimes\gg$ 
and define 
$$\bar X_j=
\begin{cases}
\hfill (X_j,0) &\text{ for }j=1,\dots,n,\\
(X_{j-n},X_{j-n}) &\text{ for }j=n+1,\dots,2n.
\end{cases} 
$$
Then $\bar X_1,\dots,\bar X_{2n}$ is a Jordan-H\"older basis 
in $\bar\gg$ and the corresponding predual for 
the coadjoint orbit $\bar\Oc\subseteq\bar\gg^*$ associated with 
the representation $\bar\pi$ is $$\bar\gg_{\bar e}=\gg_e\times\gg_e\subseteq\bar\gg,$$
where $\bar e$ is the set of jump indices for~$\bar\Oc$. 
\end{enumerate}
\end{lemma}

\begin{proof}
\eqref{pi_bar_item1} 
It is clear that the diagram is commutative, and then 
the mapping $\bar\pi$ is a representation since $\pi^{\otimes 2}$ is a representation 
and $\mu\colon G\times G\to G\ltimes G$ is a group isomorphism. 
It is well-known that the representation $\pi^{\otimes 2}$ 
is irreducible, hence $\bar\pi$ is irreducible as well. 
For the sake of completeness, we recall the corresponding reasoning. 
Let arbitrary $\Ac\in\Bc(\Sg_2(\Hc))$ satisfying 
\begin{equation}\label{com}
(\forall (g,h)\in G\ltimes G)\quad \Ac\bar\pi(g,h)=\bar\pi(g,h)\Ac. 
\end{equation}
We have to show that $\Ac$ is a scalar multiple of the identity operator on $\Sg_2(\Hc)$. 
For that purpose, let us define the operators 
$L_B,R_B\colon\Sg_2(\Hc)\to\Sg_2(\Hc)$ by $L_BX=BX$ and $R_BX=XB$ 
for $X,B\in\Bc(\Hc)$. 
Note that if $h\in G$, then $\bar\pi(\1,h)=L_{\pi(h)}$. 
It then follows by~\eqref{com} that $\Ac L_{\pi(h)}=L_{\pi(h)}\Ac$ for every $h\in G$. 
On the other hand, the representation $\pi$ is irreducible, 
the linear space $\spa\{\pi(h)\mid h\in G\}$ is dense in $\Bc(\Hc)$ 
in the strong operator topology, and then it easily follows that 
$\Ac L_B=L_B\Ac$ for every $B\in\Bc(\Hc)$. 
This property implies that there exists $A\in\Bc(\Hc)$ such that $\Ac= R_A$ 
(see for instance \cite{Ta03}). 
Now, by using \eqref{com} for $h=\1$, 
we get $\pi(g)X\pi(g)^{-1}A=\pi(g)XA\pi(g)^{-1}$ for every $g\in G$ and $X\in\Bc(\Hc)$, 
which implies $A\pi(g)=\pi(g)A$ for arbitrary $g\in G$. 
Since $\pi$ is an irreducible representation, it follows that $A$ is a scalar multiple 
of the identity operator on $\Hc$, hence $\Ac=L_A$ 
is a scalar multiple of the identity operator on $\Sg_2(\Hc)$, as we wished for.

\eqref{pi_bar_item3} 
This assertion follows by Remark~\ref{smoothness}.

\eqref{pi_bar_item4} 
It is easy to see that the sequence 
$$(0,X_1),\dots,(0,X_n),(X_1,0),\dots,(X_n,0) $$
is a Jordan-H\"older basis in the direct product $\gg\times\gg$, 
and the coadjoint orbit corresponding to the representation $\pi^{\otimes 2}\colon G\times G\to\Bc(\Sg_2(\Hc))$ 
is $\Oc\times\Oc$. 
(This follows for instance by the theorem in \S~6, Ch. II, Part 2 in \cite{Pu67}.)
Then the assertion follows by Example~\ref{sd_ex2}\eqref{sd_ex2_item4} 
along with the above Assertion~\eqref{pi_bar_item1}. 
\end{proof}

In the following definition we use an idea similar to one used in~\cite{MP09}.

\begin{definition}\label{pi_diez}
\normalfont
Let $G$ be a simply connected, nilpotent Lie group with a unitary irreducible representation 
$\pi\colon G\to\Bc(\Hc)$. 
Assume that $\Oc\subseteq\gg^*$ is the coadjoint orbit associated with 
this representation and define 
$$\pi^{\#}\colon G\ltimes G\to\Bc(L^2(\Oc)),\quad 
\pi^{\#}(\exp_G X,\exp_G Y)f=\ee^{\ie\hake{\cdot,X}}\#\ee^{\ie\hake{\cdot,Y}}\# f\#\ee^{-\ie\hake{\cdot,X}},$$
where $\#$ is the Moyal product associated with~$\pi$ 
(see Definition~\ref{moyal}). 
We note the following equivalent expression 
\begin{equation}\label{pi_diez_eq1}
(\forall X,Y\in\gg)\quad 
\pi^{\#}(\exp_{G\ltimes G}(X,Y))f=\ee^{\ie\hake{\cdot,X+Y}}\# f\#\ee^{-\ie\hake{\cdot,X}}
\end{equation}
which follows by Example~\ref{sd_ex2}\eqref{sd_ex2_item3}.
The corresponding \emph{ambiguity function} is given by 
$$\Ac^{\#}_\Phi F\colon \gg_e\times\gg_e\to\CC,\quad
(\Ac^{\#}_\Phi F)(X,Y)=(F\mid\pi^{\#}(\exp_{G\ltimes G}(X,Y))\Phi)$$
for $\Phi,F\in L^2(\Oc)$ or for a function $\Phi\in\Sc(\Oc)$ and 
a continuous antilinear functional $F\colon\Sc(\Oc)\to\CC$ 
denoted by $\Psi\mapsto (F\mid \Psi)$. 
\qed
\end{definition}

\begin{remark}\label{expl}
\normalfont
To explain the terminology of Definition~\ref{pi_diez}, 
let us see that we really have to do with the 
ambiguity function of a unitary representation. 
To this end,  recall the unitary operator $\Op^\pi\colon L^2(\Oc)\to\Sg_2(\Hc)$ 
(see e.g., Example~\ref{exs}\eqref{exs_item3})
and the representation $\bar\pi\colon G\ltimes G\to\Bc(\Sg_2(\Hc))$ from Lemma~\ref{pi_bar}. 
It follows by Definition~\ref{moyal} and Example~\ref{exs}\eqref{exs_item4} 
that the unitary operator $\Op^\pi$ intertwines $\pi^{\#}$ and $\bar\pi$, 
hence we get by Lemma~\ref{pi_bar} that $\pi^{\#}$ is also 
a unitary irreducible representation. 
It also follows that  $\gg_e\times\gg_e\subseteq\gg\ltimes\gg$ 
is a predual to the coadjoint orbit $\Oc^{\#}\subseteq(\gg\ltimes\gg)^*$ 
associated with the representation $\pi^{\#}$. 

Let us note that the space of smooth vectors for the representation $\pi^{\#}$ 
is equal to $\Sc(\Oc)$, as a consequence of Lemma~\ref{pi_bar}\eqref{pi_bar_item3}, 
since $\Op^\pi\colon\Sc(\Oc)\to\Bc(\Hc)_\infty$ is a linear isomorphism 
by Th.~4.1.4 in~\cite{Pe94}. 
\qed
\end{remark}

The next statement points out 
the representation theoretic background of the computation 
carried out in the proof of Lemma~14.5.1 in \cite{Gr01}. 

\begin{proposition}\label{exact}
Let $G$ be a simply connected, nilpotent Lie group with a unitary irreducible representation 
$\pi\colon G\to\Bc(\Hc)$.   
Pick any predual  $\gg_e\subseteq\gg$ for 
the coadjoint orbit $\Oc\subseteq\gg^*$ 
corresponding to the representation~$\pi$. 
If either $\phi_1,\phi_2,f_1,f_2\in\Hc$ 
or $\phi_1,\phi_2\in\Hc_\infty$ and $f_1,f_2\in\Hc_{-\infty}$, 
then 
$$(\forall X,Y\in\gg_e)\quad 
\Ac^{\#}_\Phi(\Wig(f_1,f_2))(X,Y)
=(\Ac_{\phi_1}f_1)(X+Y)\cdot\overline{(\Ac_{\phi_2}f_2)(X)} $$
where $\Phi:=\Wig(\phi_1,\phi_2)\in L^2(\Oc)$, while 
$\Wig(\cdot,\cdot)$ and 
$\Ac_{\phi_j}f_j\colon\gg_e\to\CC$ for $j=1,2$ 
are cross-Wigner distributions and ambiguity functions for the representation~$\pi$, respectively. 
\end{proposition}

\begin{proof}
If we denote $F=\Wig(f_1,f_2)$, then 
for arbitrary $X,Y\in\gg_e$ we have by Definition~\ref{pi_diez}, Example~\ref{sd_ex2}\eqref{sd_ex2_item3}, 
and Remark~\ref{expl},
\allowdisplaybreaks
\begin{align}
(\Ac^{\#}_\Phi F)(X,Y) 
&=(F\mid\pi^{\#}(\exp_{G\ltimes G}(X,Y))\Phi)_{L^2(\Oc)} \nonumber\\
&=(F\mid\pi^{\#}(\exp_G X,(\exp_G X)^{-1}\exp_G(X+Y))\Phi)_{L^2(\Oc)} \nonumber\\
&=(\Op^\pi(F)\mid \bar\pi(\exp_G X,(\exp_G X)^{-1}\exp_G(X+Y))\Op^\pi(\Phi))_{\Sg_2(\Hc)} \nonumber\\
&=(\Op^\pi(F)\mid\pi(\exp_G(X+Y))\Op^\pi(\Phi)\pi(\exp_G X)^{-1})_{\Sg_2(\Hc)}. \nonumber
\end{align}
On the other hand Remark~\ref{wig_ext} (particularly \eqref{wig_ext_eq1}) 
shows that 
$$\Op^\pi(F)=(\cdot\mid f_2)f_1$$ and 
$\Op^\pi(\Phi)=(\cdot\mid\phi_2)\phi_1$, whence 
$$\pi(\exp_G(X+Y))\Op^\pi(\Phi)\pi(\exp_G X)^{-1}
=(\cdot\mid\pi(\exp_G X)\phi_2)\pi(\exp_G(X+Y))\phi_1.$$ 
Then the above computation leads to the formula 
$$(\Ac^{\#}_\Phi F)(X,Y)=(\pi(\exp_G X)\phi_2\mid f_2)\cdot
(f_1\mid\pi(\exp_G(X+Y))\phi_1), $$
which is equivalent to the equation in the statement.
\end{proof}

We now prove a generalization of Th.~4.1 in \cite{To04} 
to irreducible representations of nilpotent Lie groups. 

\begin{theorem}\label{wigner_cont_th}
Let $G$ be a simply connected, nilpotent Lie group with a unitary irreducible representation 
$\pi\colon G\to\Bc(\Hc)$. 
Let $\Oc$ be the corresponding coadjoint orbit, 
pick $\phi_1,\phi_2\in\Hc_\infty\setminus\{0\}$, 
and denote $\Phi=\Wig(\phi_1,\phi_2)\in\Sc(\Oc)$. 
Assume that $\gg_e$ is a predual to the coadjoint orbit $\Oc$, 
and let $\gg_e=\gg_e^1\dotplus\gg_e^2$ be any direct sum decomposition. 

If $1\le r\le s\le\infty$ and $r_1,r_2,s_1,s_2\in[r,s]$ satisfy 
$\frac{1}{r_1}+\frac{1}{r_2}
=\frac{1}{s_1}+\frac{1}{s_2}
=\frac{1}{r}+\frac{1}{s}$, 
then the cross-Wigner distribution 
defines a continuous sesquilinear map 
$$\Wig(\cdot,\cdot)\colon M^{r_1,s_1}_{\phi_1}(\pi)\times M^{r_2,s_2}_{\phi_2}(\pi) 
\to M^{r,s}_{\Phi}(\pi^{\#}).$$
\end{theorem}

\begin{proof}
Let $f_1,f_2\in\Hc_{-\infty}$ and note that for every $X\in\gg_e$ we have 
\begin{equation}\label{wigner_cont_th_Delta}
\overline{(\Ac_{\phi_2}f_2)(X)}=\overline{(f_2\mid\pi(\exp_G X)\phi_2)}=(\Ac_{f_2}\phi_2)(-X).
\end{equation}
Therefore by Proposition~\ref{exact} we get 
\begin{equation}\label{wigner_cont_th_zero}
\Vert\Wig(f_1,f_2)\Vert_{M^{r,s}_\Phi(\pi^{\#})}=\Bigl(\int\limits_{\gg_e^2}F(Y_2)\de Y_2\Bigr)^{1/s},
\end{equation}
where 
\begin{equation}\label{wigner_cont_th_star}
F(Y_2)=\int\limits_{\gg_e^1}\Bigl(\int\limits_{\gg_e^2}\int\limits_{\gg_e^1}
\vert(\Ac_{\phi_1}f_1)(X_1+Y_1,X_2+Y_2)\cdot(\Ac_{f_2}\phi_2)(-X_1,-X_2)\vert^r
\de X_1\de X_2\Bigr)^{s/r} \de Y_1.
\end{equation}
On the other hand, it follows by Minkowski's inequality that for every measurable function 
$\Gamma\colon\gg_e^1\times\gg_e^2\times\gg_e^2\to\CC$ and every real number $t\ge1$ we have 
\begin{equation}\label{wigner_cont_th_sstar}
\Bigl(\int\limits_{\gg_e^1}\Bigl(\int\limits_{\gg_e^2}\vert\Gamma(Y_1,X_2,Y_2)\vert\de X_2\Bigr)^t 
\de Y_1\Bigr)^{1/t}
\le\int\limits_{\gg_e^2}\Bigl(\int\limits_{\gg_e^1}\vert\Gamma(Y_1,X_2,Y_2)\vert^t\de Y_1\Bigr)^{1/t}\de X_2
\end{equation}
whenever $Y_2\in\gg_e^2$. 
By \eqref{wigner_cont_th_star} and \eqref{wigner_cont_th_sstar} with $t:=s/r$ 
and 
$$\Gamma(Y_1,X_2,Y_2):=\int\limits_{\gg_e^2}
\vert(\Ac_{\phi_1}f_1)(Y_1-X_1,Y_2-X_2)\cdot(\Ac_{f_2}\phi_2)(X_1,X_2)\vert^r\de X_1$$ 
we get 
\begin{equation}\label{wigner_cont_th_ssstar}
\begin{aligned}
F(Y_2)
\le &\Bigl(\int\limits_{\gg_e^2}\Bigl(\int\limits_{\gg_e^1}\Gamma(Y_1,X_2,Y_2)^{s/r}
      \de Y_1\Bigr)^{r/s}
      \de X_2\Bigr)^{s/r} \\
    =&\Bigl(\int\limits_{\gg_e^2}\Vert\Gamma(\cdot,X_2,Y_2)\Vert_{L^{s/r}(\gg_e^1)}
      \de X_2\Bigr)^{s/r}. 
\end{aligned} 
\end{equation}
Now note that $\Gamma(\cdot,X_2,Y_2)$ is equal to the convolution product of the functions 
$\vert(\Ac_{\phi_1}f_1)(\cdot,Y_2-X_2)\vert^r$ and $\vert(\Ac_{f_2}\phi_2)(\cdot,X_2)\vert^r$. 
It follows by Young's inequality that 
$$\begin{aligned}
\Vert\Gamma(\cdot,X_2,Y_2)\Vert_{L^{s/r}(\gg_e^1)}
&\le\Vert\vert(\Ac_{\phi_1}f_1)(\cdot,Y_2-X_2)\vert^r\Vert_{L^{t_1}(\gg_e^1)}\cdot
     \Vert\vert(\Ac_{f_2}\phi_2)(\cdot,X_2)\vert^r\Vert_{L^{t_2}(\gg_e^1)} \\
&=\Vert(\Ac_{\phi_1}f_1)(\cdot,Y_2-X_2)\Vert_{L^{rt_1}(\gg_e^1)}^r\cdot
     \Vert(\Ac_{f_2}\phi_2)(\cdot,X_2)\Vert_{L^{rt_2}(\gg_e^1)}^r
\end{aligned}$$
whenever $t_1,t_2\in[1,\infty]$ satisfy $\frac{1}{t_1}+\frac{1}{t_2}=1+\frac{r}{s}$. 
By using the above inequality with $t_j=\frac{r_j}{r}$ for $j=1,2$, 
and taking into account \eqref{wigner_cont_th_ssstar}, we get 
\begin{equation}\label{wigner_cont_th_sssstar}
\begin{aligned}
F(Y_2)\le
  &\Bigl(\int\limits_{\gg_e^2}
        \Vert(\Ac_{\phi_1}f_1)(\cdot,Y_2-X_2)\Vert_{L^{rt_1}(\gg_e^1)}^r
     \Vert(\Ac_{f_2}\phi_2)(\cdot,X_2)\Vert_{L^{rt_2}(\gg_e^1)}^r
      \de X_2\Bigr)^{s/r} \\
  &=:\theta(Y_2)^{s/r}, 
\end{aligned}
\end{equation}
where $\theta(\cdot)$ is the convolution product of the functions 
$X_2\mapsto \Vert(\Ac_{\phi_1}f_1)(\cdot,X_2)\Vert_{L^{rt_1}(\gg_e^1)}^r$ 
and 
$X_2\mapsto\Vert(\Ac_{f_2}\phi_2)(\cdot,X_2)\Vert_{L^{rt_2}(\gg_e^1)}^r$. 
It follows by Young's inequality again that 
$$\begin{aligned}
\Vert\theta\Vert_{L^{s/r}(\gg_e^2)}
\le &
\Bigl(\int\limits_{\gg_e^2}\Vert(\Ac_{\phi_1}f_1)(\cdot,X_2)\Vert_{L^{rt_1}(\gg_e^1)}^r\de X_2\Bigr)^{1/m_1} \\
&\times\Bigl(\int\limits_{\gg_e^2}\Vert(\Ac_{f_2}\phi_2)(\cdot,X_2)\Vert_{L^{rt_2}(\gg_e^1)}^r\de X_2\Bigr)^{1/m_2} 
\end{aligned}$$
provided that $m_1,m_2\in[1,\infty]$ and $\frac{1}{m_1}+\frac{1}{m_2}=1+\frac{r}{s}$. 
For $m_j=\frac{s_j}{r}$, $j=1,2$, we get 
$$\Vert\theta\Vert_{L^{s/r}(\gg_e^2)}
\le (\Vert f_1\Vert_{M^{r_1,s_1}_{\phi_1}(\pi)})^r (\Vert f_2\Vert_{M^{r_2,s_2}_{\phi_2}(\pi)})^r, $$
where we also used \eqref{wigner_cont_th_Delta}. 
Then by \eqref{wigner_cont_th_zero} and \eqref{wigner_cont_th_sssstar} 
we get 
$$\Vert\Wig(f_1,f_2)\Vert_{M^{r,s}_\Phi(\pi^{\#})}\le\Vert f_1\Vert_{M^{r_1,s_1}_{\phi_1}(\pi)}\cdot
\Vert f_2\Vert_{M^{r_2,s_2}_{\phi_2}(\pi)},$$ 
and this concludes the proof. 
\end{proof}

\begin{remark}\label{wigner_cont_rem}
\normalfont
A particularly sharp version of Theorem~\ref{wigner_cont_rem} holds for $r_1=s_1$, $r_2=s_2$, and $r=s$. 
That is, let $r,r_1,r_2\in[1,\infty]$ such that $\frac{1}{r_1}+\frac{1}{r_2}=\frac{1}{r}$. 
It follows at once by Proposition~\ref{exact} that for arbitrary $f_1,f_2\in\Hc_{-\infty}$ we have 
$$\Vert\Wig(f_1,f_2)\Vert_{M^{r,r}_\Phi(\pi^{\#})}=\Vert f_1\Vert_{M^{r_1,r_1}_{\phi_1}(\pi)}\cdot
\Vert f_2\Vert_{M^{r_2,r_2}_{\phi_2}(\pi)},$$
which in turn implies that $\Wig(f_1,f_2)\in M^{r,r}_\Phi(\pi^{\#})$ 
if and only if for $j=1,2$ we have $f_j\in M^{r_j,r_j}_{\phi_j}(\pi)$. 
\qed
\end{remark}

\begin{corollary}\label{wigner_cont}
Let $G$ be a simply connected, nilpotent Lie group with a unitary irreducible representation 
$\pi\colon G\to\Bc(\Hc)$, pick $\phi_1,\phi_2\in\Hc_\infty\setminus\{0\}$, 
and denote $\Phi=\Wig(\phi_1,\phi_2)\in\Sc(\Oc)$. 
If $r,r_1,r_2\in[1,\infty]$ and $\frac{1}{r}=\frac{1}{r_1}+\frac{1}{r_2}$, 
then the cross-Wigner distribution associated with 
any predual to the coadjoint orbit of the representation~$\pi$ 
defines a continuous sesquilinear map 
$$\Wig(\cdot,\cdot)\colon M^{r_1,r_1}_{\phi_1}(\pi)\times M^{r_2,r_2}_{\phi_2}(\pi) 
\to M^{r,\infty}_{\Phi}(\pi^{\#}).$$
\end{corollary}

\begin{proof}
One can apply Theorem~\ref{wigner_cont_th} with $r_1=s_1$, $r_2=s_2$, and $s=\infty$. 
Alternatively, a direct proof proceeds as follows. 
Let $f_1,f_2\in\Hc_{-\infty}$. 
It follows by Proposition~\ref{exact} along with H\"older's inequality that 
for every $Y\in\gg_e$ we have 
$$
\Vert\Ac^{\#}_\Phi(\Wig(f_1,f_2))(\cdot,Y)\Vert_{L^r(\gg_e)}
\le\Vert\Ac_{\phi_1}f_1\Vert_{L^{r_1}(\gg_e)}\cdot 
\Vert\Ac_{\phi_2}f_2\Vert_{L^{r_2}(\gg_e)} $$
whence 
$\Vert\Wig(f_1,f_2)\Vert_{M^{r,\infty}_\Phi(\pi^{\#})}
\le\Vert f_1\Vert_{M^{r_1,r_1}_{\phi_1}(\pi)}\cdot 
\Vert f_2\Vert_{M^{r_2,r_2}_{\phi_2}(\pi)} $, and the conclusion follows. 
\end{proof}

The next corollary provides a partial generalization of Th.~4.3 in \cite{To04}. 

\begin{corollary}\label{C2}
Let $G$ be a simply connected, nilpotent Lie group with a unitary irreducible representation 
$\pi\colon G\to\Bc(\Hc)$, pick $\phi_1,\phi_2\in\Hc_\infty\setminus\{0\}$, 
and denote $\Phi=\Wig(\phi_1,\phi_2)\in\Sc(\Oc)$. 
Assume that $\gg_e$ is a predual to the coadjoint orbit $\Oc$ associated with the representation $\pi$, 
and let $\gg_e=\gg_e^1\dotplus\gg_e^2$ be any direct sum decomposition. 
If $r,s,r_1,s_1,r_2,s_2\in[1,\infty]$ satisfy the conditions  
$$r\le s, \quad r_2,s_2\in[r,s],\quad\text{and}\quad
\frac{1}{r_1}-\frac{1}{r_2}=\frac{1}{s_1}-\frac{1}{s_2}=1-\frac{1}{r}-\frac{1}{s},$$
then the following assertions hold: 
\begin{enumerate}
\item For every symbol $a\in M^{r,s}_\Phi(\pi^{\#})$ we have 
a bounded linear operator 
$$\Op^\pi(a)\colon M^{r_1,s_1}_{\phi_1}(\pi)\to M^{r_2,s_2}_{\phi_2}(\pi).$$ 
\item The linear mapping $\Op^\pi(\cdot)\colon  M^{r,s}_\Phi(\pi^{\#})\to\Bc(M^{r_1,s_1}_{\phi_1}(\pi),M^{r_2,s_2}_{\phi_2}(\pi))$ 
is continuous. 
\end{enumerate}
\end{corollary}

\begin{proof} 
For every $t\in[1,\infty]$ we are going to define $t'\in[1,\infty]$ by the equation $\frac{1}{t}+\frac{1}{t'}=1$. 
With this notation, the hypothesis implies 
$\frac{1}{r_1}+\frac{1}{r_2'}=\frac{1}{s_1}+\frac{1}{s_2'}=\frac{1}{r'}+\frac{1}{s'}$ 
and moreover $r_1,s_1,r_2',s_2'\in[r',s']$. 
Therefore we can apply Theorem~\ref{wigner_cont_th} to obtain 
\begin{equation}\label{C2_eq1}
\Vert\Wig(f_2,f_1)\Vert_{M^{r',s'}_\Phi(\pi^{\#})}\le
\Vert f_1\Vert_{M^{r_1,s_1}_{\phi_1}(\pi)}\cdot\Vert f_2\Vert_{M^{r_2',s_2'}_{\phi_2}(\pi)} 
\end{equation}
whenever $f_1,f_2\in\Hc_{-\infty}$. 

On the other hand, if $a\in M^{r,s}(\pi^{\#})$, then 
$$(\Op^\pi(a)f_1\mid f_2)=(a\mid\Wig(f_2,f_1))_{L^2(\Oc)}
=(\Ac^{\#}_\Phi a\mid\Ac^{\#}_\Phi(\Wig(f_2,f_1)))_{L^2(\gg_e\times\gg_e)}$$
by Corollary~\ref{pseudo_prop}\eqref{pseudo_prop_item2} and Proposition~\ref{orthog}\eqref{orthog_item1}.
Then H\"older's inequality for mixed-norm spaces 
(see for instance Lemma~11.1.2(b) in \cite{Gr01}) 
shows that 
$$\begin{aligned}
\vert(\Op^\pi(a)f_1\mid f_2)\vert
&\le 
\Vert\Ac^{\#}_\Phi a\Vert_{L^{r,s}(\gg_e\times\gg_e)} \cdot \Vert\Ac^{\#}_\Phi(\Wig(f_2,f_1))\Vert_{L^{r',s'}(\gg_e\times\gg_e)}\\ 
&=\Vert a\Vert_{M_\Phi^{r,s}(\pi^{\#})} \cdot 
\Vert\Wig(f_2,f_1)\Vert_{M_\Phi^{r',s'}(\pi^{\#})} \\
&\le\Vert a\Vert_{M_\Phi^{r,s}(\pi^{\#})} \cdot
\Vert f_1\Vert_{M^{r_1,s_1}_{\phi_1}(\pi)}\cdot\Vert f_2\Vert_{M^{r_2',s_2'}_{\phi_2}(\pi)},
\end{aligned}$$
where the latter inequality follows by \eqref{C2_eq1}. 
Now the assertion follows by a straightforward argument 
that uses the duality of the mixed-norm spaces 
(see Lemma~11.1.2(d) in \cite{Gr01}). 
\end{proof}

\begin{corollary}\label{C3}
If $G$ be a simply connected, nilpotent Lie group with a unitary irreducible representation 
$\pi\colon G\to\Bc(\Hc)$, 
then the following assertions hold whenever $\Phi=\Wig(\phi_1,\phi_2)$ with $\phi_1,\phi_2\in\Hc_\infty\setminus\{0\}$: 
\begin{enumerate}
\item For every $a\in M^{\infty,1}_\Phi(\pi^{\#})$ we have 
$\Op^\pi(a)\in\Bc(\Hc)$. 
\item The linear mapping $\Op^\pi(\cdot)\colon M^{\infty,1}_\Phi(\pi^{\#})\to\Bc(\Hc)$ is continuous. 
\end{enumerate}
\end{corollary}

\begin{proof}
This is the special case of Corollary~\ref{C2} with with 
$r_1=s_1=r_2=s_2=2$, $r=1$, and $s=\infty$, since  
Example~\ref{mod_L2} shows that $M^{2,2}(\pi)=\Hc$. 
\end{proof}

We conclude this section by a sufficient condition for 
a pseudo-differential operator to belong to the trace class. 
In the special case of the Schr\"odinger representation of a Heisenberg group, 
a proof for this result can be found for instance in \cite{Gr96} or \cite{GH99}.

\begin{proposition}\label{trace_class}
Let $G$ be a simply connected, nilpotent Lie group with a unitary irreducible representation 
$\pi\colon G\to\Bc(\Hc)$, pick $\phi_1,\phi_2\in\Hc_\infty$ with $\Vert\phi_1\Vert=\Vert\phi_2\Vert=1$, 
and denote $\Phi=\Wig(\phi_1,\phi_2)\in\Sc(\Oc)$. 
Then for every symbol $a\in M^{1,1}_\Phi(\pi^{\#})$ we have 
$\Op^\pi(a)\in\Sg_1(\Hc)$ and $\Vert\Op^\pi(a)\Vert_1\le\Vert a\Vert_{M^{1,1}_\Phi(\pi^{\#})}$. 
\end{proposition}

\begin{proof}
For arbitrary $a\in\Sc'(\Oc)$ we have by Corollary~\ref{orthog_distrib}\eqref{orthog_distrib_item1} 
and Remark~\ref{expl}, 
$$a=\iint\limits_{\gg_e\times\gg_e}(\Ac_\Phi^{\#}a)(X,Y)\cdot\pi^{\#}(\exp_{G\ltimes G}(X,Y))\Phi\de X\de Y, $$
whence by Corollary~\ref{pseudo_prop}
\begin{equation}\label{trace_class_eq1}
\Op^\pi(a)=\iint\limits_{\gg_e\times\gg_e}(\Ac_\Phi^{\#}a)(X,Y)\cdot\Op^\pi(\pi^{\#}(\exp_{G\ltimes G}(X,Y))\Phi)
\de X\de Y
\end{equation}
where the latter integral is weakly convergent in $\Lc(\Hc_\infty,\Hc_{-\infty})$. 
On the other hand, for arbitrary $X,Y\in\gg_e$ we get by \eqref{pi_diez_eq1} 
and Corollary~\ref{pseudo_prop}\eqref{pseudo_prop_item3}
$$\begin{aligned}
\Op^\pi(\pi^{\#}(\exp_{G\ltimes G}(X,Y))\Phi)
&=\pi(\exp_G(X+Y))\circ\Op^\pi(\Phi)\circ\pi(\exp_G X)^{-1} \\
&=(\cdot\mid\pi(\exp_G X)\phi_2)\pi(\exp_G(X+Y))\phi_1. 
\end{aligned}$$
In particular, $\Op^\pi(\pi^{\#}(\exp_{G\ltimes G}(X,Y))\Phi)\in\Sg_1(\Hc)$ 
and 
$$\Vert\Op^\pi(\pi^{\#}(\exp_{G\ltimes G}(X,Y))\Phi)\Vert_1
=\Vert\pi(\exp_G(X+Y))\phi_1\Vert\cdot\Vert\pi(\exp_G X)\phi_2\Vert=1.$$ 
It then follows that the integral in \eqref{trace_class_eq1} 
is absolutely convergent in $\Sg_1(\Hc)$ for $a\in M^{1,1}_\Phi(\pi^{\#})$ 
and moreover we have 
$$\Vert\Op^\pi(a)\Vert_1\le\iint\limits_{\gg_e\times\gg_e}\vert(\Ac_\Phi^{\#}a)(X,Y)\vert
\de X\de Y=\Vert a\Vert_{M^{1,1}_\Phi(\pi^{\#})}$$
which concludes the proof. 
\end{proof}

\section{The case of square-integrable representations}

In this section we focus on square-integrable representations of nilpotent Lie groups.
A discussion of the crucial role of these representations along with many examples 
can be found for instance in the monograph \cite{CG90}. 

\subsection{Independence of the modulation spaces on the window vectors}

\begin{lemma}\label{right_invar}
Let $G_1$ and $G_2$ be unimodular Lie groups and assume that 
we have a group homomorphism $\alpha\colon G_1\to\Aut G_2$, $g_1\mapsto\alpha_{g_1}$. 
Consider the semidirect product $G=G_1\ltimes_\alpha G_2$ 
and for every $h\in G$ and $\phi\colon G\to\CC$ 
define $R_h\phi\colon G\to\CC$, $(R_h\phi)(g)=\phi(gh)$. 
Fix $r,s\in[1,\infty]$ and consider the mixed-norm space 
$L^{r,s}(G)$ consisting of the equivalence classes of 
functions $\phi\colon G\to\CC$ such that 
$$\Vert\phi\Vert_{L^{r,s}(G)}:=
\Bigl(\int\limits_{G_2}\Bigl(\int\limits_{G_1}\vert\phi(g_1,g_2)\vert^r 
\de g_1\Bigr)^{s/r}\de g_2\Bigr)^{1/s}<\infty, $$
with the usual conventions if $r$ or $s$ is infinite. 
Then the space $L^{r,s}(G)$ is invariant under the right-translation operator $R_h$ for every $h\in G$, and the mapping 
$$\rho\colon G\to\Bc(L^{r,s}(G)),\quad h\mapsto R_h\vert_{L^{r,s}(G)}$$
is a strongly continuous representation of the Lie group $G$ by isometries 
on the Banach space $L^{r,s}(G)$.
\end{lemma}

\begin{proof}
Let $\phi\colon G\to\CC$ be any measurable function 
and $h=(h_1,h_2)\in G$. 
We have $(R_h\phi)(g_1,g_2)=\phi(g_1h_1,\alpha_{h_1^{-1}}(g_2)h_2)$.  
Since the group $G_1$ is unimodular, it then follows that 
for every $g_2\in G_2$ we have 
$$\int\limits_{G_1}\vert(R_h\phi)(g_1,g_2)\vert^r \de g_1
=\int\limits_{G_1}\vert\phi(g_1h_1,\alpha_{h_1^{-1}}(g_2)h_2)\vert^r \de g_1
=\int\limits_{G_1}\vert\phi(g_1,\alpha_{h_1^{-1}}(g_2)h_2)\vert^r \de g_1. $$
Now, by integrating on $G_2$ both extreme terms in above equality 
and taking into account that $G_2$ is unimodular, we get 
$$(\forall h\in G)\quad 
\Vert R_h\phi\Vert_{L^{r,s}(G)}=\Vert\phi\Vert_{L^{r,s}(G)}. $$
With this equality at hand, it is straightforward to 
prove all the assertions in the statement just as in the classical case 
when $r=s$.
\end{proof}

\begin{remark}\label{integral}
\normalfont
In the setting of Lemma~\ref{right_invar}, 
for every $\psi\in L^1(G)$ we can define 
the bounded linear operator $\rho(\psi)\colon L^{r,s}(G)\to L^{r,s}(G)$ by 
$$(\forall \chi\in L^{r,s}(G))\quad 
\rho(\psi)\chi=\int\limits_G \psi(h)R_h\chi\de h. $$
Then for every $\chi\in L^{r,s}(G)$ we have 
$$(\rho(\psi)\chi)(g)=\int\limits_G \chi(gh)\psi(h)\de h 
\text{ for a.e. }g\in G$$
and 
$$\Vert\rho(\psi)\phi\Vert_{L^{r,s}(G)}\le C\Vert\phi\Vert_{L^{r,s}(G)}, $$
where $C$ denotes the norm of the operator~$\rho(\psi)$, 
hence $C\le\Vert\psi\Vert_{L^1(G)}$. 
\qed
\end{remark}

We are now ready to prove a theorem that covers many cases 
when the modulation spaces for square-integrable representations 
do not depend on the choice of a window function. 
The second stage in the proof is inspired by the methods 
of the theory of coorbit spaces 
(see \cite{FG88}, \cite{FG89a}, \cite{FG89b}, and also 
the proof of Prop.~11.3.2(c) in \cite{Gr01}).

\begin{theorem}\label{indep_sq}
Let $G_1$ and $G_2$ be simply connected, nilpotent Lie groups 
and a unipotent homomorphism $\alpha\colon G_1\to\Aut G_2$. 
Define $G=G_1\ltimes_\alpha G_2$ and assume that the center $\zg$ 
of $\gg$ satisfies the condition 
\begin{equation}\label{indep_sq_eq1}
\zg=(\zg\cap\gg_1)+(\zg\cap\gg_2).
\end{equation} 
Assume the irreducible representation $\pi\colon G\to\Bc(\Hc)$ 
is square integrable modulo the center~of $G$, and 
pick any Jordan-H\"older basis in $\gg$ such that for 
the corresponding predual $\gg_e$ for the coadjoint orbit associated with $\pi$ 
we have $\gg_e=(\gg_e\cap\gg_1)+(\gg_e\cap\gg_2)$.

Then the modulation spaces for the representation $\pi$ 
with respect to the decomposition 
$\gg_e\simeq(\gg_e\cap\gg_1)\times(\gg_e\cap\gg_2)$
are independent on the choice of a window vector 
$\phi\in\Hc_\infty\setminus\{0\}$. 
\end{theorem}

\begin{proof} 
The proof has two stages. 

$1^\circ$ 
For the sake of simplicity let us identify the Lie group $G_j$ to its Lie algebra~$\gg_j$ by means of 
the exponential map $\exp_{G_j}$, so that $G_j$ will be just $\gg_j$ with the group operation $\ast$ 
defined by the Baker-Campbell-Hausdorff series. 
Let $Z$ be the center of $G$, whose Lie algebra $\zg$ is the center of $\gg$. 
Then we have a linear isomorphism $\gg_e\simeq\gg/\zg$, $X\mapsto X+\zg$, 
and we shall endow $\gg_e$ with the Lie algebra structure which makes this 
map into an isomorphism of Lie algebras. 

If we define $G_e:=G/Z$, then $G_e$ is a connected, simply connected nilpotent Lie group, 
whose Lie algebra is just $\gg_e$. 
Let $\ast_e$ denote the multiplication in $G_e$, which is just 
the Baker-Campbell-Hausdorff multiplication in $\gg_e$. 

Now use assumption \eqref{indep_sq_eq1} to see that 
if $(Y_1,Y_2)\in\zg\subseteq\gg=\gg_1\ltimes_{\dot\alpha}\gg_2$, 
then $(Y_1,0),(0,Y_2)\in\zg$. 
Now formula~\eqref{sd_bracket} shows that 
for every $(X_1,X_2)\in\gg$ we have 
$0=[(X_1,X_2),(Y_1,0)]=([X_1,Y_1],-\dot\alpha(Y_1)X_2)$, 
hence $Y_1$ belongs to the center $\zg_1$ of $\gg_1$ and 
$\dot\alpha(Y_1)=0$. 
This shows that the closed subgroup $Z_1:=Z\cap G_1$ is contained in the center of $G_1$ and satisfies 
\begin{equation}\label{indep_sq_proof_eq1}
Z_1\subseteq\Ker\alpha. 
\end{equation}
Also
$0=[(X_1,X_2),(0,Y_2)]=(0,\dot\alpha(X_1)Y_2+[X_2,Y_2])$ 
for every $(X_1,X_2)\in\gg$, 
whence we see that $Y_2$ belongs both to the center $\zg_2$ of $G_2$ 
and to $\Ker(\dot\alpha(X_1))$ for arbitrary $X_1\in\gg_1$. 
Therefore the closed subgroup $Z_2:=Z\cap G_2$ is contained in the center of $G_2$ 
and we have 
\begin{equation}\label{indep_sq_proof_eq2}
(\forall g_1\in G_1)\quad  \alpha_{g_1}(Z_2)\subseteq Z_2. 
\end{equation}
It follows by \eqref{indep_sq_proof_eq1} and \eqref{indep_sq_proof_eq2} 
that the group homomorphism $\alpha\colon G_1\to\Aut G_2$ induces 
a group homomorphism $\bar\alpha\colon G_1/Z_1\to\Aut(G_2/Z_2)$ 
and we have the isomorphisms of Lie groups 
$$G_e\simeq G/Z\simeq(G_1/Z_1)\ltimes_{\bar\alpha}(G_2/Z_2).$$
Moreover $Z\simeq Z_1\times Z_2$. 

$2^\circ$ We now come back to the proof. 
Fix $r,s\in[1,\infty]$ and let $\phi_1,\phi_2\in\Hc_\infty$ 
be any window functions with $\Vert\phi_1\Vert=\Vert\phi_2\Vert=1$. 
For $j=1,2$ and every $f\in\Hc_{-\infty}$ we have by Corollary~\ref{orthog_distrib} 
\allowdisplaybreaks
\begin{align}
(\Ac_{\phi_2}f)(X)
&=(f\mid\pi(\exp_G X)\phi_2) \nonumber\\
&=\int\limits_{\gg_e}\chi(Y)
(\pi(\exp_G Y)\phi_1\mid\pi(\exp_G X)\phi_2)\de Y \nonumber\\
&=\int\limits_{\gg_e}\chi(Y)
(\phi_1\mid\pi(\exp_G((-Y)\ast X))\phi_2)\de Y \nonumber\\
&=\int\limits_{\gg_e}\chi(Y)\ee^{\ie\alpha(-Y,X)}
(\phi_1\mid
 \pi(\exp_G ((-Y)\ast_e X))\phi_2)\de Y, \nonumber\\
&=\int\limits_{\gg_e}\chi(Y)\ee^{\ie\alpha(-Y,X)}
(\Ac_{\phi_2}\phi_1)((-Y)\ast_e X)\de Y, \nonumber\\
&=\int\limits_{\gg_e}\chi(X\ast_e Y)
\ee^{\ie\alpha((-Y)\ast_e(-X),X)}
(\Ac_{\phi_2}\phi_1)(-Y)\de Y, \nonumber
\end{align}
for every $X\in\gg_e$, 
where $\chi:=\Ac_{\phi_1}f\in L^{r,s}(\gg_e\times\gg_e)$ 
and $\alpha\colon\gg_e\times\gg_e\to\RR$ is a suitable polynomial function 
defined in terms of the central character of the representation~$\pi$ 
(see e.g., \cite{Ma07}). 
Now note that $\Ac_{\phi_2}\phi_1\in\Sc(\gg_e)$ 
by Corollary~\ref{orthog_distrib}\eqref{orthog_distrib_item3}. 
It then follows by Lemma~\ref{right_invar} and Remark~\ref{integral} 
that there exists a constant $C>0$ such that for every $f\in\Hc_{-\infty}$ we have 
$\Vert\Ac_{\phi_2}f\Vert_{L^{r,s}(\gg_e\times\gg_e)}\le \Vert\Ac_{\phi_1}f\Vert_{L^{r,s}(\gg_e\times\gg_e)}$. 
Thus we get the continuous inclusion map $M^{r,s}_{\phi_1}(\pi)\hookrightarrow M^{r,s}_{\phi_2}(\pi)$. 
Now the conclusion follows by interchanging $\phi_1$ and $\phi_2$. 
\end{proof}

The previous theorem allows us to omit the window vector in the notation for  modulation spaces 
associated to square-integrable representations. 

\begin{example}\label{indep_sq_ex}
\normalfont
Theorem~\ref{indep_sq} applies to a wide variety of situations. 
Let us mention here just a few of them: 
\begin{enumerate}
\item In the case of the Schr\"odinger representation 
of the Heisenberg group $\HH_{2n+1}=\RR^n\ltimes\RR^{n+1}$ 
we recover the well-known property that the classical modulation spaces 
used in the time-frequency analysis 
are independent on the choice of a window function 
(see for instance Prop.~11.3.2(c) in \cite{Gr01}). 
\item We shall see below (see subsection~\ref{subsect3.3}) that one can give sufficient 
conditions for  
the continuity of the operators constructed by the Weyl-Pedersen calculus for
the square-integrable representation $\pi\colon G\to\Bc(\Hc)$ 
by using spaces of symbols which are modulation spaces 
$M^{r,s}(\pi^{\#})\subseteq\Sc'(\Oc)$. 
Here $\pi^{\#}\colon G\ltimes G\to\Bc(L^2(\Oc))$ 
is in turn a square integrable representation to which Theorem~\ref{indep_sq} 
applies and ensures that the corresponding modulation spaces 
do not depend on the choice of a window function. 
\end{enumerate}
\qed
\end{example}

\subsection{Covariance properties of the Weyl-Pedersen calculus}

We now record the covariance property for the cross-Wigner distributions 
and its consequence for the Weyl-Pedersen calculus. 
In the very special case of the Schr\"odinger representation for the Heisenberg group 
we recover a classical fact (see e.g., \cite{Gr01}). 

\begin{theorem}\label{cov}
Assume that 
%$\gg_1$ is equal to the center of $\gg$, $\hake{\xi_0,X_1}=1$, and 
the representation $\pi\colon G\to\Bc(\Hc)$ associated with $\Oc$ 
is square integrable modulo the center~of $G$. 
Then the following assertions hold: 
\begin{enumerate}
\item\label{cov_item1} 
For every $f,h\in\Hc$ and  $X\in\gg$  we have 
$$\Wig(\pi(\exp_G X)f,\pi(\exp_G X)h)(\xi)=\Wig(f,h)(\xi\circ\ee^{\ad_{\gg}X}) 
\text{ for a.e. }\xi\in\Oc.$$
\item\label{cov_item2} 
For every symbol $a\in\Sc'(\Oc)$ and arbitrary $g\in G$ we have 
$$\Op(a\circ\Ad_G^*(g^{-1})\vert_{\Oc})=\pi(g)\Op(a)\pi(g)^{-1}. $$
\end{enumerate}
\end{theorem}

\begin{proof} 
\eqref{cov_item1}
Note that 
the following assertions hold: 
\allowdisplaybreaks 
\begin{align}
\label{cov_proof_eq1} 
(\forall X\in\zg)\quad  & \pi(\exp_G X)=\ee^{\ie\hake{\xi_0,X}}\id_{\Hc},\\
\label{cov_proof_eq2} 
(\forall \xi\in\Oc)\quad  &\xi\vert_{\zg}=\xi_0\vert_{\zg}, \\
\label{cov_proof_eq3} & \xi_0\vert_{\gg_e}=0.
\end{align}
Also, it easily follows by Definition~\ref{amb} that for arbitrary $f,h\in\Hc$ we have 
$$\Wig(f,h)(\xi)=\int\limits_{\gg_e}\ee^{\ie\hake{\xi,X}}(f\mid\pi(\exp_G X)h)\de X
\text{ for a.e. }\xi\in\Oc. $$
It then follows that for arbitrary $X_0\in\gg_e$ and a.e.\ $\xi\in\Oc$ we have
$$\begin{aligned}
\Wig(\pi(\exp_G X_0)f,&\pi(\exp_G X_0)h)(\xi) \\
&=\int\limits_{\gg_e}\ee^{\ie\hake{\xi,X}}(f\mid\pi((\exp_G(-X_0))(\exp_G X)(\exp_G X_0))h)\de X \\
&=\int\limits_{\gg_e}\ee^{\ie\hake{\xi,X}}(f\mid\pi(\exp_G (\ee^{\ad_{\gg}(-X_0)}X))h)\de X.
\end{aligned}$$ 
If we denote by $\pr_{\zg}\colon\gg\to\zg$ the natural projection 
corresponding to the direct sum decomposition $\gg=\zg\dotplus\gg_e$, then we have 
for every $X\in\gg_e$, 
$$\ee^{\ad_{\gg}(-X_0)}X=\ee^{\ad_{\gg_e}(-X_0)}X+\pr_{\zg}(\ee^{\ad_{\gg}(-X_0)}X), $$
where we have endowed $\gg_e$ with the Lie algebra structure 
which makes the linear isomorphism $\gg_e\simeq\gg/\zg$ into an isomorphism of Lie algebras 
(see also~\cite{Ma07}). 
Therefore, by using \eqref{cov_proof_eq1} and \eqref{cov_proof_eq3}, we get 
$$(\forall X\in\gg_e)\quad 
\pi((\exp_G (\ee^{\ad_{\gg}(-X_0)}X)))
=\ee^{\ie\hake{\xi_0,\ee^{\ad_{\gg}(-X_0)}X}}\pi(\exp_G (\ee^{\ad_{\gg_e}(-X_0)}X))$$
and then the above computation leads to  
$$\begin{aligned}
\Wig(\pi(\exp_G X_0)f,&\pi(\exp_G X_0)h)(\xi) \\
&=\int\limits_{\gg_e}\ee^{\ie\hake{\xi,X}}\ee^{-\ie\hake{\xi_0,\ee^{\ad_{\gg}(-X_0)}X}}
(f\mid\pi(\exp_G (\ee^{\ad_{\gg_e}(-X_0)}X))h)\de X\\
&=\int\limits_{\gg_e}\ee^{\ie\hake{\xi,X}}\ee^{-\ie\hake{(\ee^{\ad_{\gg}(-X_0)})^*\xi_0,X}}
(f\mid\pi(\exp_G (\ee^{\ad_{\gg_e}(-X_0)}X))h)\de X\\
&=\int\limits_{\gg_e}\ee^{\ie\hake{\xi,\ee^{\ad_{\gg_e}X_0}Y}}
\ee^{-\ie\hake{(\ee^{\ad_{\gg}(-X_0)})^*\xi_0,\ee^{\ad_{\gg_e}X_0}Y}}
(f\mid\pi(\exp_G Y)h)\de Y
\end{aligned}$$ 
where we used the change of variables $X\mapsto Y=\ee^{\ad_{\gg_e}(-X_0)}X$, 
which is a measure-preserving diffeomorphism since $\gg_e$ is a nilpotent Lie algebra. 
Now note that by using \eqref{cov_proof_eq2} we get for a.e.\ $\xi\in\Oc$ and every $Y\in\gg_e$, 
\allowdisplaybreaks
\begin{align}
\hake{\xi,\ee^{\ad_{\gg_e}X_0}Y}-\hake{(\ee^{\ad_{\gg}(-X_0)})^*\xi_0,&\ee^{\ad_{\gg_e}X_0}Y} \nonumber\\
=&\hake{\xi,\ee^{\ad_{\gg}X_0}Y}-\hake{\xi,\pr_{\zg}(\ee^{\ad_{\gg}X_0}Y)}  \nonumber\\
&-\hake{\xi_0,\ee^{\ad_{\gg}(-X_0)}(\ee^{\ad_{\gg}X_0}Y-\pr_{\zg}(\ee^{\ad_{\gg}X_0}Y))} \nonumber\\
=&\hake{\xi,\ee^{\ad_{\gg}X_0}Y}-\hake{\xi_0,\ee^{\ad_{\gg}X_0}Y}  \nonumber\\
&-\hake{\xi_0,Y}+\hake{\xi_0,\ee^{\ad_{\gg}X_0}Y} \nonumber\\
=&\hake{\xi,\ee^{\ad_{\gg}X_0}Y} \nonumber
\end{align}
since $\hake{\xi_0,Y}=0$ by \eqref{cov_proof_eq3}. 
Thus the conclusion follows by the formula we had obtained above for 
$\Wig(\pi(\exp_G X_0)f,\pi(\exp_G X_0)h)(\xi)$, 
and this completes the proof for $X\in\gg_e$. 
Then the formula extends to arbitrary $X\in\gg$ by using 
the fact that $\gg=\gg_e\dotplus\zg$ and taking into account~\eqref{cov_proof_eq1}. 

\eqref{cov_item2} 
If $a\in\Sc(\Oc)$, then for every $f,\phi\in\Hc$ we have 
$$\begin{aligned}
(\Op(a\circ\Ad_G^*(g^{-1})\vert_{\Oc})\phi\mid f)_{\Hc}
&=(a\circ\Ad_G^*(g^{-1})\vert_{\Oc}\mid\Wig(f,\phi))_{L^2(\Oc)} \\
&=(a\mid\Wig(f,\phi)\circ\Ad_G^*(g)\vert_{\Oc})_{L^2(\Oc)} \\
&=(a\mid\Wig(\pi(g)^{-1}f,\pi(g)^{-1}\phi))_{L^2(\Oc)} \\
&=(\Op(a)\pi(g)^{-1}\phi\mid\pi(g)^{-1}f)_{\Hc} \\
&=(\pi(g)\Op(a)\pi(g)^{-1}\phi\mid f)_{\Hc}, 
\end{aligned}$$
where the first and the fourth equalities follow by 
Corollary~\ref{pseudo_prop}\eqref{pseudo_prop_item2}, 
the second equality is a consequence of the fact that 
the coadjoint action preserves the Liouville measure on $\Oc$, 
while the third equality follows by Assertion~\eqref{cov_item1} 
which we already proved. 

Thus we obtained the conclusion for $a\in\Sc(\Oc)$, 
and then it can be easily extended by duality to any $a\in\Sc'(\Oc)$ 
by using equation~\eqref{duality} in Definition~\ref{calc_def}. 
\end{proof}

\subsection{Continuity of Weyl-Pedersen calculus}\label{subsect3.3}

\begin{lemma}\label{sq_pi_bar}
Let $G$ be any Lie group with a unitary irreducible representation 
$\pi\colon G\to\Bc(\Hc)$  and  
$$\bar\pi\colon G\ltimes G\to\Bc(\Sg_2(\Hc)),\quad \bar\pi(g,h)T=\pi(gh)T\pi(g)^{-1}.$$
 If $G$ is a unimodular group and $\pi$ is square integrable modulo the center of~$G$, 
then $\bar\pi$ is square integrable modulo the center of $G\ltimes G$. 
\end{lemma}

\begin{proof}
If $\pi$ is square integrable modulo the center $Z$ of $G$, 
then there is $\phi_0\in\Hc\setminus\{0\}$ such that 
the function $gZ\mapsto \vert(\pi(g)\phi_0\mid\phi_0)\vert$ 
is square integrable on $G/Z$. 
Let us define the rank-one projection $T_0=(\cdot \mid\phi_0)\phi_0$ 
Then we have 
$$\begin{aligned}
\iint\limits_{(G\ltimes G)/(Z\times Z)}\vert(\bar\pi(g,h)T_0\mid T_0)\vert^2\de g\de h 
&=\int\limits_{G/Z}\Bigl(\int\limits_{G/Z}\vert(\pi(gh)T_0\pi(g)^{-1}\mid T_0)\vert^2\de h\Bigr)\de g  \\
&= \int\limits_{G/Z}\Bigl(\int\limits_{G/Z}\vert(\pi(h)T_0\pi(g)^{-1}\mid T_0)\vert^2\de h\Bigr)\de g.
  \end{aligned}
$$
Since $T_0=(\cdot \mid\phi_0)\phi_0$, we get 
$\pi(h)T_0\pi(g)^{-1}=(\cdot\mid\pi(g)\phi_0)\pi(h)\phi_0$, and then 
$$(\pi(h)T_0\pi(g)^{-1}\mid T_0)=(\pi(h)\phi_0\mid\phi_0)\cdot(\phi_0\mid\pi(g)\phi_0).$$
Therefore 
$$\iint\limits_{(G\ltimes G)/(Z\times Z)}\vert(\bar\pi(g,h)T_0\mid T_0)\vert^2\de g\de h 
=\Bigl(\int\limits_{G/Z}\vert(\pi(g)\phi_0\mid \phi_0)\vert^2\de h\Bigr)^2 $$
hence the function $(g,h)(Z\times Z)\mapsto \vert(\bar\pi(g,h)T_0\mid T_0)\vert$ 
is square integrable on the quotient group $(G\ltimes G)/(Z\times Z)$, 
and this concludes the proof since $Z\times Z$ is the center of $G\ltimes G$ 
(see Example~\ref{sd_ex2}). 
\end{proof}

\begin{remark}\label{sq_expl}
\normalfont
Assume that $\pi\colon G \to \Bc(\Hc)$ is a square-integrable representation of a simply connected, nilpotent Lie group, with the corresponding coadjoint orbit $\Oc\subseteq \gg^*$. 
 Recall the representation $\pi^\#\colon G\ltimes G\to\Bc(L^2(\Oc))$ from Definition~\ref{pi_diez} 
(see also Remark~\ref{expl}). 
The assumption that $\pi$ is square integrable modulo the center of $G$ 
implies by Theorem~\ref{cov}\eqref{cov_item2} 
that $\pi^\#$ is given by 
$$\pi^{\#}\colon G\ltimes G\to\Bc(L^2(\Oc)),\quad 
\pi^{\#}(g,\exp Y)f=(\ee^{\ie\hake{\cdot,Y}}\# f)\circ\Ad_G^*(g^{-1})|_{\Oc}. $$
Since the unitary operator $\Op^\pi\colon L^2(\Oc)\to\Sg_2(\Hc)$  intertwines $\pi^{\#}$ and $\bar\pi$, 
we get by Lemma~\ref{sq_pi_bar} that $\pi^{\#}$ is also 
a unitary irreducible representation which is square integrable 
modulo the center $Z\times Z$ of $G\ltimes G$. 
\qed
\end{remark}

\begin{corollary}\label{sq_wigner_cont}
Let $G$ be a simply connected, nilpotent Lie group with a unitary irreducible representation 
$\pi\colon G\to\Bc(\Hc)$ which is square integrable modulo the center of~$G$. 
If $r,r_1,r_2\in[1,\infty]$ and $\frac{1}{r}=\frac{1}{r_1}+\frac{1}{r_2}$, 
then the cross-Wigner distribution associated with 
any predual to the coadjoint of the representation $\pi$ 
defines a continuous sesquilinear map 
$$\Wig(\cdot,\cdot)\colon M^{r_1,r_1}(\pi)\times M^{r_2,r_2}(\pi) 
\to M^{r,\infty}(\pi^{\#}).$$
\end{corollary}

\begin{proof} 
Firstly use Corollary~\ref{wigner_cont}.
Then the conclusion follows since both $\pi$ and $\pi^{\#}$ 
are square integrable representations (see also Remark~\ref{sq_expl}), 
hence Theorem~\ref{indep_sq} shows that 
the topologies of the modulation spaces 
$M^{r_1,r_1}(\pi)$, $M^{r_2,r_2}(\pi)$, and 
$M^{r,\infty}(\pi^{\#})$ can be defined by any special choice 
of window functions. 
\end{proof}

\begin{corollary}\label{sq_C1}
If $G$ be a simply connected, nilpotent Lie group with a unitary irreducible representation 
$\pi\colon G\to\Bc(\Hc)$ which is square integrable modulo the center of~$G$, 
then the cross-Wigner distribution associated with 
any predual to the coadjoint of the representation $\pi$ 
defines a continuous sesquilinear map 
$$\Wig(\cdot,\cdot)\colon \Hc\times\Hc 
\to M^{1,\infty}(\pi^{\#}).$$
\end{corollary}

\begin{proof}
Just apply Corollary~\ref{sq_wigner_cont} with $r_1=r_2=2$ and $r=1$; 
and recall from Example~\ref{mod_L2} that $M^{2,2}(\pi)=\Hc$. 
\end{proof}

In the special case of the Schr\"odinger representation for the Heisenberg group, 
the following corollary recovers the assertion of Th.~1.1 in \cite{GH99} 
concerning the boundedness of pseudo-differential operators 
defined by the classical Weyl-H\"ormander  calculus on~$\RR^n$. 

\begin{corollary}\label{sq_C2}
Let $G$ be a simply connected, nilpotent Lie group with a unitary irreducible representation 
$\pi\colon G\to\Bc(\Hc)$ which is square integrable modulo the center of~$G$. 
If $r,r',r_1,r_2\in[1,\infty]$ satisfy the equations  $\frac{1}{r}=\frac{1}{r_1}+\frac{1}{r_2}=1-\frac{1}{r'}$, 
then the following assertions hold: 
\begin{enumerate}
\item For every symbol $a\in M^{r',1}(\pi^{\#})$ we have 
a bounded linear operator 
$$\Op^\pi(a)\colon M^{r_1,r_1}(\pi)\to M^{r_2,r_2}(\pi).$$ 
\item The linear mapping $\Op^\pi(\cdot)\colon M^{r',1}(\pi^{\#})\to\Bc(M^{r_1,r_1}(\pi),M^{r_2,r_2}(\pi))$ 
is continuous. 
\end{enumerate}
\end{corollary}

\begin{proof}
Firstly use Corollary~\ref{C2}. 
Then the conclusion follows since 
Theorem~\ref{indep_sq} shows that 
the topologies of the modulation spaces involved in the statement 
can be defined by any special choice 
of window functions. 
\end{proof}

\begin{corollary}\label{sq_C3}
If $G$ be a simply connected, nilpotent Lie group with a unitary irreducible representation 
$\pi\colon G\to\Bc(\Hc)$ which is square integrable modulo the center of~$G$, 
then the following assertions hold: 
\begin{enumerate}
\item For every $a\in M^{\infty,1}(\pi^{\#})$ we have 
$\Op^\pi(a)\in\Bc(\Hc)$. 
\item The linear mapping $\Op^\pi(\cdot)\colon M^{\infty,1}(\pi^{\#})\to\Bc(\Hc)$ is continuous. 
\end{enumerate}
\end{corollary}

\begin{proof}
This is the special case of Corollary~\ref{sq_C2} with with 
$r_1=r_2=2$ and $r=1$, since  
Example~\ref{mod_L2} shows that $M^{2,2}(\pi)=\Hc$. 
\end{proof}

\section{Schr\"odinger representations of the Heisenberg groups}\label{Sect4}

We show in the present section that, in the special case of the Heisenberg group, 
the modulation spaces of symbols defined in our paper are in fact nothing else than 
the modulation spaces widely used in time-frequency analysis.  

\subsection*{Schr\"odinger representations}
Let $\Vc$ be a finite-dimensional vector space 
endowed with a nondegenerate bilinear form denoted by $(p,q)\mapsto p\cdot q$.
The corresponding \emph{Heisenberg algebra} 
${\mathfrak h}_{\Vc}=\Vc\times\Vc\times{\mathbb R}$ 
is the Lie algebra with the bracket 
$$[(q,p,t),(q',p',t')]=[(0,0,p\cdot q'-p'\cdot q)]. $$
The \emph{Heisenberg group} 
${\mathbb H}_{\Vc}$ is just ${\mathfrak h}_{\Vc}$ thought of as a group 
with the multiplication~$\ast$ 
defined by 
$$X\ast Y=X+Y+\frac{1}{2}[X,Y]. $$
The unit element is $0\in{\mathbb H}_{\Vc}$ 
and the inversion mapping given by $X^{-1}:=-X$.
The \emph{Schr\"odinger representation} is the unitary representation 
$\pi_{\Vc}\colon{\mathbb H}_{\Vc}\to{\mathcal B}(L^2(\Vc))$ 
defined by  
\begin{equation}\label{sch_eq}
(\pi_{\Vc}(q,p,t)f)(x)={\rm e}^{{\rm i}(p\cdot x+\frac{1}{2}p\cdot q+t)}f(x+q) 
\quad\text{ for a.e. }x\in\Vc
\end{equation}
for arbitrary $f\in L^2(\Vc)$ and $(q,p,t)\in{\mathbb H}_{\Vc}$. 
This is a square-integrable representation and the corresponding coadjoint orbit of $\HH_{\Vc}$ is 
\begin{equation}\label{sch_eq1}
\Oc=\{\xi\colon\hg_{\Vc}\to\RR\text{ linear}\mid \xi(0,0,1)=1\}.
\end{equation}
Let $\xi_0\in\Oc$ be the functional satisfying $\xi_0(q,p,0)=0$ 
for every $q,p\in\Vc$.
If we denote $\dim\Vc=n$, then any basis $\{x_1,\dots,x_n\}$ in $\Vc$ 
naturally gives rise to the Jordan-H\"older basis 
$$(x_1,0,0),\dots,(x_n,0,0),(0,x_1,0),\dots,(0,x_n,0),(0,0,1) $$
in $\hg_{\Vc}$ and the corresponding predual of $\Oc$ is 
$$(\hg_{\Vc})_e=\Vc\times\Vc\times\{0\}. $$
For the sake of an easier comparison with the previously obtained results 
we shall denote $G=\HH_{\Vc}$ and $\gg=\hg_{\Vc}$ from now on, 
and in particular we shall denote $\gg_e=(\hg_{\Vc})_e$. 

\subsection*{Computing the Moyal product representation}
Recall from \cite{Ma07} that for every $f,h\in\Sc(\Oc)$ we have 
\begin{equation*}
(\forall\xi\in\Oc)\quad
(f\# h)(\xi)=\iint\limits_{\gg_e\times\gg_e}
\ee^{\ie\hake{\xi,X+Y}}\ee^{(\ie/2)\hake{\xi_0,[X,Y]}}
\widehat{f}(X)\widehat{h}(Y)\de X\de Y.
\end{equation*}
It then follows by a duality argument that for every $f\in\Sc(\Oc)$ 
and $V\in\gg$ we have 
$$(f\#\ee^{-\ie\hake{\cdot,V}})(\xi)=
\int\limits_{\gg_e}\ee^{\ie\hake{\xi,X-V}}
\ee^{(\ie/2)\hake{\xi_0,[X,-V]}}
\widehat{f}(X)\de X,
 $$
whence 
\begin{equation}\label{diez_r}
(\forall V\in\gg)(\forall\xi\in\Oc)\quad (f\#\ee^{-\ie\hake{\cdot,V}})(\xi)
=\ee^{-\ie\hake{\xi,V}}f(\xi+(1/2)\xi_0\circ\ad_{\gg}V).
\end{equation}
Since $\overline{f\# h}=\bar{h}\#\bar{f}$, we also get 
\begin{equation}\label{diez_l}
(\forall V\in\gg)(\forall\xi\in\Oc)\quad (\ee^{\ie\hake{\cdot,V}}\# f)(\xi)
=\ee^{\ie\hake{\xi,V}}f(\xi+(1/2)\xi_0\circ\ad_{\gg}V).
\end{equation}
Now for arbitrary $X,Y\in\gg$, $f\in\Sc(\Oc)$, and $\xi\in\Oc$ we get 
\allowdisplaybreaks
\begin{align}
(\ee^{\ie\hake{\cdot,X+Y}}\# f\#\ee^{-\ie\hake{\cdot,X}})
=&\ee^{\ie\hake{\xi,X+Y}}
(f\#\ee^{-\ie\hake{\cdot,X}})(\xi+(1/2)\xi_0\circ\ad_{\gg}(X+Y)) \nonumber\\
=&\ee^{\ie\hake{\xi,X+Y}}
\ee^{-\ie\hake{\xi+(1/2)\xi_0\circ\ad_{\gg}(X+Y),X}} \nonumber\\
&\times f(\xi+(1/2)\xi_0\circ\ad_{\gg}(X+Y)+(1/2)\xi_0\circ\ad_{\gg}X) \nonumber \\
=&\ee^{\ie\hake{\xi,Y}}
\ee^{(\ie/2)\hake{\xi_0,[X,Y]}} \nonumber\\
&\times f(\xi+(1/2)\xi_0\circ\ad_{\gg}(X+(1/2)Y)). \nonumber
\end{align}
By taking into account \eqref{pi_diez_eq1}, we now see that 
the unitary irreducible representation 
$\pi_{\Vc}^{\#}\colon G\ltimes G\to\Bc(L^2(\Oc))$ is given by 
\begin{equation}\label{sch_diez}
(\pi_{\Vc}^{\#}(\exp_{G\ltimes G}(X,Y))f)(\xi)
=\ee^{\ie(\hake{\xi,Y}+\hake{\xi_0,[X,Y]}/2)} 
f(\xi+\xi_0\circ\ad_{\gg}(X+(1/2)Y)) 
\end{equation}
where the latter equation follows by \eqref{sch_eq1}. 

\subsection*{Abstract unitary equivalence}
Denote the center of $G$ by $Z$, with the corresponding Lie algebra $\zg$. 
The above formula yields
$\exp_{G\ltimes G}(\zg\times\{0\})\subseteq\Ker\,\pi_{\Vc}^{\#}$, 
hence we get a unitary irreducible representation 
$\overline{\pi_{\Vc}^{\#}}\colon (G\ltimes G)/(Z\times\{\1\})\to\Bc(L^2(\Oc))$. 
Also note that there exist the natural isomorphisms of Lie groups 
\begin{equation}\label{sch_isom}
(G\ltimes G)/(Z\times\{\1\})\simeq(G/Z)\ltimes G\simeq\HH_{\Vc\times\Vc}.
\end{equation}
By specializing \eqref{sch_diez} for $X,Y\in\zg$ we can see that 
the representation $\overline{\pi_{\Vc}^{\#}}$ has the same central character as 
the Schr\"odinger representation of the Heisenberg group $\HH_{\Vc\times\Vc}$, 
hence they are unitarily equivalent to each other, as a consequence of the Stone-von Neumann theorem. 

\subsection*{Specific unitary equivalence}
Alternatively, we can exhibit an explicit unitary equivalence as follows. 
Let us consider the affine isomorphism 
$\Oc\to(\Vc\times\Vc)^*$, $\xi\mapsto\xi\vert_{\Vc\times\Vc\times\{0\}}$, 
and the natural embedding $\Vc\times\Vc\simeq\Vc\times\Vc\times\{0\}\hookrightarrow\hg_{\Vc}$. 
Now for $X,Y\in\Vc\times\Vc$ and $t\in\RR$ we have 
$(X,Y,t)\in\HH_{\Vc\times\Vc}\simeq (G\ltimes G)/(Z\times\{\1\})$ 
(see \eqref{sch_isom}) hence 
$$\begin{aligned}
(\overline{\pi_{\Vc}^{\#}}(\exp_{\HH_{\Vc\times\Vc}}(X,Y,t))f)(\xi)
=&(\pi_{\Vc}^{\#}(\exp_{G\ltimes G}((X,0),(Y,t)))f)(\xi) \\
=&\ee^{\ie(\hake{\xi,Y}+t+\omega_{\xi_0}(X,Y)/2)} 
f(\xi+\xi_0\circ\ad_{\gg}(X+(1/2)Y))
\end{aligned}$$
where 
$\omega_{\xi_0}(V,W):=\hake{\xi_0,[V,W]}$ whenever $V,W\in\Vc\times\Vc\hookrightarrow\gg$. 
Thence we get 
$$(\overline{\pi_{\Vc}^{\#}}(\exp_{\HH_{\Vc\times\Vc}}(X-(1/2)Y,Y,t))f)(\xi)
=\ee^{\ie(\hake{\xi,Y}+t+\omega_{\xi_0}(X,Y)/2)} 
f(\xi+\xi_0\circ\ad_{\gg}X). $$
Note that 
$\psi\colon\hg_{\Vc\times\Vc}\to\hg_{\Vc\times\Vc}$, $(X,Y,t)\mapsto(X-(1/2)Y,Y,t)$ 
is an automorphism of the Heisenberg algebra $\hg_{\Vc\times\Vc}$, 
hence, by denoting by $\Psi\colon\HH_{\Vc\times\Vc}\to\HH_{\Vc\times\Vc}$ 
the corresponding automorphism of the Heisenberg group $\HH_{\Vc\times\Vc}$, 
we get 
\begin{equation*}
(\widetilde{\pi}(\exp_{\HH_{\Vc\times\Vc}}(X,Y,t))f)(\xi)=
\ee^{\ie(\hake{\xi,Y}+\omega_{\xi_0}(X,Y)/2+t)} 
f(\xi+\xi_0\circ\ad_{\gg}X)
\end{equation*}
where $\widetilde{\pi}:=\overline{\pi_{\Vc}^{\#}}\circ\Psi$ 
is again a representation of the Heisenberg group $\HH_{\Vc\times\Vc}$. 
Then for arbitrary $V\in\Vc\times\Vc$ we get %by \eqref{sch_tilde}
\begin{equation}\label{sch_tilde}
\begin{aligned}
(\widetilde{\pi}(\exp_{\HH_{\Vc\times\Vc}}(X,Y,t))f)(\xi_0+\xi_0\circ\ad_{\gg}V)
=& \ee^{\ie(\omega_{\xi_0}(V,Y)+\omega_{\xi_0}(X,Y)/2+t)} \\
&\times f(\xi_0+\xi_0\circ\ad_{\gg}(V+X)).
\end{aligned}
\end{equation}
Now let us define the affine isomorphism
$$A\colon\Vc\times\Vc\to\Oc,\quad V\mapsto\xi_0+\xi_0\circ\ad_{\gg}V$$
and consider the unitary operator 
$U\colon L^2(\Vc\times\Vc)\to L^2(\Oc)$, $f\mapsto f\circ A^{-1}$. 
It follows by the above equation that if we define the Heisenberg group $\HH_{\Vc\times\Vc}$ 
by using the nondegenerate bilinear map 
$$(\Vc\times\Vc)\times(\Vc\times\Vc)\to\RR,\quad (V,W)\mapsto-\omega_{\xi_0}(V,W),$$
then the unitary operator $U$ intertwines the representation 
$\widetilde{\pi}\colon\HH_{\Vc\times\Vc}\to\Bc(L^2(\Oc))$ 
and the Schr\"odinger representation 
$\pi_{\Vc\times\Vc}\colon\HH_{\Vc\times\Vc}\to\Bc(L^2(\Vc\times\Vc))$. 
In other words, the operator $U$ induces a unitary equivalence 
of the representation $\overline{\pi_{\Vc}^{\#}}$ with the representation 
$\pi_{\Vc\times\Vc}\circ\Psi^{-1}$. 

\subsection*{Determining the modulation spaces of symbols}
It follows by the above discussion 
that the operator $U$ induces isomorphisms between 
the modulation spaces for the representations 
$$\pi_{\Vc}^{\#}\colon G\ltimes G\to\Bc(L^2(\Oc))
\quad\text{and}\quad
\pi_{\Vc\times\Vc}\circ\Psi^{-1}\colon \HH_{\Vc\times\Vc}\to\Bc(L^2(\Vc\times\Vc)).$$ 
Now note that for arbitrary $r,s\in[1,\infty]$ 
we have $M^{r,s}(\pi_{\Vc\times\Vc}\circ\Psi^{-1})=M^{r,s}(\pi_{\Vc\times\Vc})$
since the norm of 
any measurable function $f\colon \Vc\times\Vc\to\CC$ in $L^{r,s}(\Vc\times\Vc)$ 
is equal to the norm of the function $(X,Y)\mapsto f(X+(1/2)Y,Y)$ in the same space. 
Therefore \emph{the operator $f\mapsto f\circ A^{-1}$ actually induces an isomorphism 
from the modulation space $M^{r,s}(\pi_{\Vc}^{\#})$ onto 
the modulation space $M^{r,s}(\pi_{\Vc\times\Vc})$ of the Schr\"odinger representation 
$\pi_{\Vc\times\Vc}\colon \HH_{\Vc\times\Vc}\to\Bc(L^2(\Vc\times\Vc))$}. 
Finally, recall that $M^{r,s}(\pi_{\Vc\times\Vc})=M^{r,s}(\Vc\times\Vc)$, 
where the latter is just the classical modulation space on $\Vc\times\Vc$ 
as used for instance in \cite{Gr01}.

\end{document}